\documentclass{siamart190516}
\usepackage{a4,latexsym,amssymb,graphicx,rotating,amsmath,subfigure}
\usepackage[normalem]{ulem}

\usepackage{booktabs}
\usepackage{enumitem}

\newcommand{\IR}{{\mathbb{R}}}

\newcommand{\eqdef}{\stackrel{\rm def}{=}}
\newcommand{\req}[1]{(\ref{#1})}
\newcommand{\frozen}{{\sc FAR2}}
\newcommand{\frozenso}{{\sc FAR2-SO}}
\newcommand{\frozenr}{{\sc FAR2-rk}}
\newcommand{\frozenp}{{\sc FAR2-pk}}

\newcommand{\galr}{{\sc AR2-rqs}}
\newcommand{\galg}{{\sc AR2-glrt}}
\newcommand{\arc}{{\sc AR2}}

\usepackage{pgfplots}
\pgfplotsset{compat=1.13}

\usepackage{algpseudocode}

\newtheorem{remark}{Remark}[section]
\usepackage{todonotes}

\newcommand{\revdone}[1]{\textcolor{black}{#1}}
\newcommand{\REV}[1]{\textcolor{black}{#1}}

\title{Regularized methods via cubic model subspace minimization for nonconvex optimization}

\author{Stefania Bellavia\thanks{Corresponding author, Dipartimento di Ingegneria Industriale, Universit\`a degli Studi di Firenze, Viale Morgagni, 40/44, 50134 Florence, Italy.
Members of the INdAM Research Group GNCS. Emails: {\tt \{stefania.bellavia,margherita.porcelli\}@unifi.it}} \and
Davide Palitta\thanks{Dipartimento di Matematica, (AM)$^2$,
Alma Mater Studiorum - Universit\`a di Bologna, Piazza di Porta San Donato 5,
 40126 Bologna, Italia. Members of the INdAM Research Group GNCS.  Emails:
{\tt  \{davide.palitta,valeria.simoncini\}@unibo.it}}\and{\,}
 Margherita Porcelli$^{*,}$\thanks{ISTI--CNR, Via Moruzzi 1, Pisa, Italia}
 \and Valeria Simoncini$^{\dagger,}$\thanks{IMATI-CNR, Via Ferrata 5/A, Pavia, Italia}}

\date{}

\begin{document}
%\today 

\maketitle

\begin{abstract} Adaptive cubic regularization methods for solving nonconvex problems need the efficient computation of  the trial step, involving the minimization of a cubic model.  
We propose a new approach in which this model is minimized in a low  dimensional subspace that, in contrast to classic approaches,  is reused for a number of iterations.  Whenever the trial step produced by the low-dimensional minimization process is unsatisfactory, we employ a regularized Newton step   whose regularization parameter is a by-product of  the model minimization over the  low-dimensional subspace.
We show that the  worst-case complexity of classic cubic regularized methods is preserved, despite the possible regularized Newton steps. We  focus on the large class of problems  for which (sparse) direct linear system solvers are available and provide
 several experimental results showing the very large gains of our new approach when compared to standard implementations of  adaptive cubic regularization methods based on direct linear solvers. Our first choice as projection space for the low-dimensional model minimization is the  polynomial Krylov subspace; nonetheless, we also explore the use of rational Krylov subspaces in case where the polynomial ones lead to less competitive numerical results.
 \end{abstract}

\begin{keywords}
Adaptive regularization methods, nonconvex optimization, secant methods, Krylov subspaces, worst-case iteration complexity
\end{keywords}

\begin{AMS}
49M37, %Numerical methods based on nonlinear programming
65K05, %Numerical mathematical programming methods
68W40 %	Analysis of algorithms

\end{AMS}
\vskip 10pt

%%%%%%%%%%%%%%%%%%%%%%%%%%%%%%%%%%%%%%%

\section{Introduction}
We address the numerical solution of possibly nonconvex, unconstrained optimization problems 
 of the form
 \begin{equation}\label{eq:pb}
 \min_{x\in \mathbb{R}^n} f(x),
\end{equation}
where the objective function $f:\mathbb{R}^n\rightarrow \mathbb{R}$ is supposed to be twice-continuously differentiable and bounded from below. 
To attack (\ref{eq:pb}) we consider  second order adaptive regularization methods (\arc). These are well established, globally convergent  variants of the Newton method for~\eqref{eq:pb}, 
where the step length  is implicitly controlled.
This feature is achieved by  adding a cubic term in the classic quadratic Taylor model to penalize long steps \cite[Section 3.3]{book_compl}. 
More precisely, at a generic \arc\ iteration $k$, the trial step is computed by  (approximately) solving the subproblem
\begin{equation}\label{eq:subcub}
\min_{s \in \IR^n} m_k(s),
\end{equation}
where $m_k(s)$ is the following {\em cubic regularized model} of the objective function  $f$ around the current point $x_k$
\begin{equation}\label{eq:cubmod}
 m_k(s) \eqdef f (x_k ) + s^T \nabla f(x_k) + \frac{1}{2} s^T \nabla^2 f(x_k) s +\frac{1}{3} \sigma_k \|s\|_2^3 = T_k(s) +\frac{1}{3} \sigma_k \|s\|_2^3.
\end{equation}
The scalar   $\sigma_k > 0$ is the regularization parameter and $T_k(s)$ is the second-order Taylor-series expansion of $f$ at $x_k$. 

For nonconvex problems, \arc\ exhibits a worst-case  ${\cal O}(\epsilon^{-3/2})$  iteration-com\-plexity to find an $\epsilon$-approximate  first-order minimizer, i.e., a point $x_\epsilon$ such that 
\begin{equation}\label{fomin}
\|\nabla f(x_\epsilon)\|\le \epsilon,
\end{equation}
and a worst-case ${\cal O}({\rm max}(\epsilon^{-3/2}, \epsilon_H^{-2}))$  iteration-complexity to find 
 an $(\epsilon,\epsilon_H)$-approx\-imate second-order minimizer, i.e.,  a point $x_{\epsilon,\epsilon_H}$ such that
\begin{equation}\label{fomin2}
\|\nabla f(x_{\epsilon,\epsilon_H})\|\le \epsilon\quad \mbox{and} \quad \lambda_{\min}
(\nabla^2 f(x_{\epsilon,\epsilon_H}))\ge  -\epsilon_H.
\end{equation}
\arc\ is optimal in terms of complexity when compared to standard second-order
methods, such as trust-region and Newton methods, whose iteration-complexity  to compute an $\epsilon$-approximate  first-order minimizer
amounts to ${\cal O}(\epsilon^{-2})$
\cite[Chapter~3] {book_compl}.

The main computational challenge of \arc\ procedures is the construction of an approximate solution to (\ref{eq:subcub}).
In the current literature this task is performed either  through the solution of
the so-called ``secular'' equation
\cite[Chapter~9] {book_compl}, or by approximately minimizing the cubic model over a sequence of nested Krylov subspaces  generated at each nonlinear iteration, onto which the secular equation is projected and solved \cite[Chapter~10] {book_compl}. 
In the first case a sequence of shifted linear systems
has to be solved \revdone{at each nonlinear iteration}.

Here we develop a hybrid procedure that combines both the approaches mentioned above for the computation of the trial step. 
We propose a predictor-corrector scheme, where the prediction step performs the minimization of the cubic model in a low dimensional Krylov subspace, whereas the corrector step, if needed, is a regularized Newton step
whose regularization parameter is provided by the prediction step.
The subspace 
 is  kept \emph{frozen} during the iterations as long as possible; whenever it fails to provide a sufficiently good regularization parameter, the space is \emph{refreshed}, namely a new space is computed. If this new subspace is still unable to provide a good regularization parameter,  then an approximate minimizer of the cubic method in the full space is computed by the standard secular-equation-based procedure \cite[Chapter~9] {book_compl}. We expect to use this latter step only occasionally.

Our novel scheme is tailored to problems where a factorization of $\nabla^2 f(x_k) + \lambda I$ for a given $\lambda$ is both feasible and efficient, along the lines of \cite{grt10}. 
This is the case for small to medium-sized problems, and also for the large-scale sparse setting; see, e.g,~\cite{ScottTuma2023}. 
We chose to employ sparse
direct methods for such systems as iterative procedures are
highly problem dependent and generally requires a suitable preconditioner.
At a first glance, adopting direct solvers for the solution of linear systems may seem in contrast with the construction of Krylov subspaces. Nonetheless, these two strategies are synergic ingredients in our 
nonlinear optimization solver.  
Indeed, the minimization of the model over a  low-dimensional and possibly frozen Krylov subspace allows us to adaptively adjust the regularization parameter for the  
regularized Newton step and preserve the complexity of the original \arc, despite a regularized Newton step is possibly employed.

The main practical advantage of our new scheme is the remarkable reduction in the overall number of factorizations of the Hessian matrix, 
when compared to standard secular-equation-based algorithms as those based on~\cite{grt10}. More precisely, our numerical experiments show that the approximation subspace is refreshed only few times (roughly in the 1.8\% of the total number of nonlinear iterations) and the secant method is seldom invoked; see section \ref{Numerical experiments}. As a consequence, a significant number of nonlinear iterations sees 
the computation of at most a single factorization. 

Our approach is very general as its derivation does not depend on the chosen reduction space. Nonetheless, the performance can be influenced by this selection. Our first choice is the (classic) polynomial Krylov subspace, which is cheap to construct and performed rather well in our extensive numerical experimentation on 
different datasets. 
We also considered using rational Krylov subspaces, which showed particularly good performance for sequences of slowly varying or severely ill-conditioned subproblems.

\subsection*{Related works}\label{sec:biblio}
The efficient solution of the subproblem (\ref{eq:subcub}) in AR2-like methods has been the subject of intense research in the last 20 years; see \cite[Chapters 9-10]{book_compl} and references therein.
As already mentioned,  state-of-the-art methods can be mainly divided into two classes: Approaches based on the approximate minimization of the regularized model using a Lanczos process \cite{cgt09,sigma,gouldSim20,Jia22,Lieder20}
and schemes based on the 
secular formulation exploiting matrix factorizations \cite{grt10}.  
In \cite{cgt09, sigma, gouldSim20} the cubic model is minimized over a sequence of nested Krylov subspaces generated by the 
Lanczos procedure. In \cite{Jia22} it is observed that this approach may  produce a large dimensional subspace when applied to ill-conditioned problems and a nested restarting version for the large scale and ill-conditioned case is proposed.  Both  \cite{Jia22}  and \cite{Lieder20} reduce the minimization subproblem  to a suitable generalized eigenvalue problem.

The algorithm proposed in~\cite{ARCqK} concurrently solves a set of systems with the shifted Hessian matrix 
for a predefined large number of shifts
 by means of a hybrid procedure employing the secular equation combined with the Lanczos method for indefinite linear systems.

 A different approach is taken in \cite{cd_SIOPT19}, where the gradient descent is used to approximate the global minimizer of the cubic model; the worst-case iteration complexity is studied. 

More recently, several works have focused on the development of
modified Newton methods which exhibit the \arc-like worst-case complexity (or nearly so); see e.g. \cite{birgin_martinez, curtis19,gst_yet23}.  
 In \cite{curtis19} a general framework covering several algorithms,  like \arc\ and the modified trust-region algorithm named {\sc TRACE} \cite{trace}, is given. At each iteration the trial step and the multiplier estimate are calculated as an approximate   solution of  a constrained optimization subproblem whose objective function is a  regularized Newton model.
 A second-order algorithm  alternating between a regularized Newton and negative curvature steps is given in \cite{gst_yet23}, while  \cite{birgin_martinez} incorporates cubic descent into a quadratic regularization framework. 
 
 Similarly to the methods in \cite{birgin_martinez,curtis19,gst_yet23}, also our novel scheme employs quadratic-regularization variants  preserving the complexity results of the cubic regularization.  On the other hand, 
 in contrast to what is done in, e.g.,~\cite{birgin_martinez,curtis19},  we employ the regularized Newton step only when the step provided by the subspace minimization is not satisfactory.
  Moreover, the regularization term is a byproduct of the subspace minimization process, and the computation of the smallest eigenvalue of the Hessian matrix is never required when approximate first-order optimality points are sought.
 We also aim at 
reducing the computational cost associated to the step computation.
To this end, we make use of low dimensional Krylov subspaces whose possibly expensive basis construction is carried out only a handful of times throughout all the nonlinear iterations. Indeed, once a Krylov subspace is constructed at a certain nonlinear iteration $k$, we reuse it in the subspace minimization as long as possible thus fully capitalizing on the computational efforts performed for its construction. Such a strategy significantly differs from the ones proposed in, e.g.,~\cite{ARCqK, Jia22, gst_yet23,Lieder20}. 
 
\subsection*{Outline}
Here is a synopsis of the paper. In
section~\ref{pre} we present some preliminary background about the secular formulation of~\eqref{eq:subcub} (section~\ref{sec:sec}) and the \arc\ method (section~\ref{The AR2 method}).
Section~\ref{The new algorithm} is devoted to present the main contribution of this paper, namely the novel solution process for~\eqref{eq:pb}. 
 In section~\ref{Complexity} we study the first-order complexity of our algorithm,  and show that  the optimal worst-case iteration complexity of \arc\ is maintained. 
In section~\ref{sec:appr} we discuss the choice of the Krylov approximation space. 
Numerical results displaying the potential of our novel approach are presented in section~\ref{Numerical experiments} while in section~\ref{Conclusions} we draw some conclusions.
Finally, Appendix~\ref{appendixA} contains a variant of the proposed algorithm, specifically designed to find approximate second-order optimality points, and the corresponding complexity analysis, while  Appendix~\ref{sec:opm} collects the complete results obtained during our numerical testing.

\subsection*{Notation}
Throughout the paper we adopt the following notation.  At the $k$-th nonlinear iteration, the gradient and the Hessian of $f$ evaluated at the current iterate $x_k$ are denoted
as $g_k = \nabla f(x_k)$ and $H_k = \nabla^2 f(x_k)$, respectively. The symbol $I$ is used for the $n\times n$ identity matrix, and $\|\cdot \|$ denotes the Euclidean norm of a vector.

\section{Preliminaries} \label{pre}
This section provides a brief description of the features of the cubic regularized minimization problem (\ref{eq:subcub}) in terms of its secular formulation 
and a quick review of the main steps of the \arc\ method.

\subsection{The ``secular'' formulation}\label{sec:sec}
The secular formulation of the subproblem (\ref{eq:subcub}) is defined by exploiting
the characterization of the global minimizer of $m_k(s)$ reported below. For the sake of simplicity, in this section we omit the subscript $k$.

\begin{theorem}\cite[Theorem 8.2.8]{book_compl}
Any global minimizer $s^*$ of (\ref{eq:subcub}) satisfies
\begin{equation}\label{shift}
(H + \lambda^* I) s^* = - g,    
\end{equation}
where $H + \lambda^* I$ is positive semidefinite and
$$\lambda^* = \sigma \|s^*\|.$$
If $H + \lambda^* I$ is positive definite, then $s^*$ is unique.
\end{theorem}

Let $\lambda_1$ be the leftmost eigenvalue of $H$. 
In the so-called ``easy case'', $\lambda^* > \lambda_S \eqdef 
\max \{ 0, -\lambda_1\}$
and a solution to (\ref{eq:subcub}) can be computed by solving the scalar {\em secular equation}
\begin{equation}\label{eq:sec}
 \phi_R(\lambda; g, H, \sigma) \eqdef \|(H + \lambda I)^{-1}g\| - \frac{\lambda}{\sigma}=0.
\end{equation}
It can be proved that the Newton and the secant methods applied to (\ref{eq:sec}) rapidly converge to $\lambda^*$
starting from any approximation in $(\lambda_S, + \infty)$ \cite{cgt09,book_compl,grt10}.
Whenever the computation of $\lambda_1$ is prohibitive and the computation of appropriate starting guesses is thus not guaranteed, suitable estimations to $\lambda_1$ can be derived as done in~\cite{grt10}.

As opposite to the ``easy case'', the ``hard case'' takes place when $\lambda_1<0$ and $g$ is orthogonal to the $\lambda_1$-eigenspace of $H$, namely $v^T g =0$ for all $v$ in $\{v:\,(H-\lambda_1I)v=0\}$. In this case, $\lambda^* = -\lambda_1$ and there is no
solution to the secular equation in  $(\lambda_S, + \infty)$. Sophisticated numerical strategies have to be employed to compute the global minimizer $s^*$ in this case; see, e.g.,~\cite[Section 9.3.4]{book_compl}. 

The theorem below summarizes the form of the global 
minimizer of (\ref{eq:subcub}) in both the easy and the  hard cases.

\begin{theorem}\cite[Corollary 8.3.1]{book_compl}\label{teo:hc}
Any global minimizer of (\ref{eq:subcub}) can be expressed as
\begin{equation}
s^* = \left \{
 \begin{array}{ll}
  -(H + \lambda^* I)^{-1} g \ (uniquely) & \mbox{ if } \lambda^* > -\lambda_1, \\
  -(H + \lambda^* I)^\dagger g+\alpha v_1 & \mbox{ if } \lambda^* = -\lambda_1, \\
 \end{array} \right .
\end{equation}
where $(H + \lambda^* I)^\dagger$ denotes the pseudoinverse of $H + \lambda^* I$, $\lambda^* = \sigma \|s^*\|$, $\lambda_1$ is the leftmost eigenvalue of $H$, $v_1$
is any corresponding eigenvector, and $\alpha$ is one of the two roots of 
$\|-(H + \lambda^* I)^\dagger g + \alpha v_1 \| = \lambda^*/\sigma.$
\end{theorem}

\subsection{The \arc\ method}\label{The AR2 method}
The \arc\ method is an iterative procedure where, at iteration $k$, a trial step $s_k$ is computed by approximately minimizing the cubic model~\eqref{eq:cubmod}.  Algorithm \ref{AR2_algo} summarizes a possible implementation of the \arc\ algorithm; see also~\cite[Algorithm 3.3.1]{book_compl}.
Theorem \ref{teo:hc} determines the exact minimizer of $m_k(s)$ but, in practice, 
only an approximate minimizer $s_k$ satisfying conditions~\req{eq:dec1}
and~\req{eq:dec2} in Algorithm~\ref{AR2_algo} is indeed necessary; see, e.g.,~\cite[p.65]{book_compl}. 

Given the trial step $s_k$,  the trial point $x_k+s_k$ is then used to compute the ratio
\begin{equation}\label{eq_rho}
 \rho_k = \frac{f(x_k) - f(x_k+s_k)}{T_k(0)-T_k(s_k)},
\end{equation}
with $T_k$ as in (\ref{eq:cubmod}).
 If $\rho_k\ge \eta_1$, with $\eta_1\in(0,1)$, then the trial point is accepted, 
 the iteration is declared  successful and  the regularization parameter $\sigma_k$ is possibly decreased. Otherwise, an unsuccessful iteration occurs: The point $x_k+s_k$ is rejected and the regularisation parameter  is increased.

\begin{algorithm}[t]
\caption{The adaptive-regularization algorithm with a second order model (\arc) algorithm 
\cite{book_compl} \label{AR2_algo}}
\begin{algorithmic}[1]
\Require  An initial point $x_0$; accuracy threshold $\epsilon \in (0,1)$;
an initial regularization parameter $\sigma >0$ are given as well as constants 
$\eta_1, \eta_2, \gamma_1, \gamma_2,  \theta_1, \sigma_{\min} $ that satisfy
$$
\sigma_{\min} \in (0,\sigma_0],\quad \theta_1 >0,\quad 0<\eta_1\le\eta_2<1,\quad 
0<\gamma_1 < 1 < \gamma_2.
$$

\State Compute $f(x_0), g_0= \nabla f(x_0), H_0 = \nabla^2 f(x_0)$, and set $k=0$.

 \State{\bf Step 1: Test for termination.} If $\|g_k\| \le \epsilon$, terminate.

 \State{\bf Step 2: Step computation.}
  Compute a step $s_k$ such that 
   \begin{equation}\label{eq:dec1}
  m_k(s_k) < m_k(0)
 \end{equation}
 and \begin{equation}\label{eq:dec2}
     \|\nabla m_k(s_k) \| \le \frac{1}{2}\theta_1 \|s_k\|^2
     \end{equation}
\State {\bf Step 3: Acceptance of the trial point.} Compute $f(x_k+s_k)$ and 
the ratio $\rho_k$ given in
\eqref{eq_rho}.

\State If $\rho_k \ge \eta_1$, then define $x_{k+1}= x_k + s_k$ and compute $g_{k+1}=\nabla f(x_{k+1})$
and $H_{k+1} = \nabla^2 f(x_{k+1})$.
Otherwise define $x_{k+1}= x_k$.

\State {\bf Step 4: Regularization parameter update.} Compute
   \begin{equation}\label{sig}
   \sigma_{k+1} \in \left \{
   \begin{array}{lll}
    [\max(\sigma_{\min}, \gamma_1 \sigma_k), \sigma_k  ]                   & \mbox{ if } \rho_k\geq \eta_2
      \\
   \sigma_k &\mbox{ if }\rho_k\in [\eta_1,\eta_2) 
      \\
   \gamma_2 \sigma_k  & \mbox{ otherwise }
         \end{array}
   \right . 
   \end{equation}
   \State Increment $k$ by one and {\bf goto} Step 1.

\end{algorithmic}
\end{algorithm}

\section{The new \frozen\ algorithm}\label{The new algorithm}
In this section we introduce our subspace cubic approach
that we name the Frozen AR2 (\frozen) algorithm. The  main idea behind \frozen\ is to construct {a low-dimensional subspace} $\mathcal K$ at the initial iteration, compute a minimizer to the cubic model projected onto this subspace, and then keep using such subspace also in the subsequent nonlinear iterations.
The minimizer of $m_k(s)$ over  $\mathcal K$ is cheaply computed by solving the \emph{projected} secular equation, that is equation~\req{eq:sec} where the quantities involved $H$ and $g$ consist of the projection of $H_k$ and $g_k$ onto $\mathcal{K}$, respectively. 
If \eqref{eq:dec2} in Algorithm~\ref{AR2_algo} is not satisfied by such minimizer\footnote{Condition \eqref{eq:dec1} is naturally satisfied as the null vector lies in such subspace.}, then a further  regularized Newton-like step is performed where the employed regularization parameter $\widehat\lambda_k$ is a by-product of the minimization process over the current  subspace. 
The current subspace $\mathcal{K}$ is discarded and a new, more informative approximation space is generated  whenever either the computed $\widehat\lambda_k$
does not provide a regularized Newton step $s_k$ such that
\begin{equation}\label{posedf_dir}
s_k^T(H_k+\widehat\lambda_k I)s_k>0,
\end{equation}
or $s_k$ fails to satisfy 
the following condition:
\begin{equation}\label{bound_p} 
    C_{{\rm low}}\|\widehat s_k\|\le  \|s_k\|\le C_{{\rm up}} \|\widehat s_k\|,
\end{equation}
where $\widehat s_k$ is the minimizer of the cubic model onto the current subspace and $C_{{\rm low}}$ and $C_{{\rm up}}$ are given positive constants. 

Notice that condition \eqref{posedf_dir} is significantly less restrictive than requiring  $H_k+\widehat\lambda_k I$ positive definite \cite{gst_yet23}. Moreover, condition (\ref{bound_p}) will be crucial to analyze the complexity of \frozen\ in section \ref{Complexity} and is similar to the one adopted in \cite[Eq. (2.1)]{curtis19}. 
Note that condition (\ref{bound_p}) is not restrictive as it only requires that the ratio $\|s_k\|/\|\widehat s_k\|$ is lower bounded away from zero and upper bounded.

\begin{algorithm}[t!b]
\caption{The Frozen AR2 (\frozen) algorithm} \label{fAR2_refresh_algo}
\begin{algorithmic}[1]
\Require  An initial point $x_0$; accuracy threshold $\epsilon \in (0,1)$;
an initial regularization parameter $\sigma_0 >0$ are given, and constants
$\eta_1, \eta_2, \gamma_1, \gamma_2,  \theta_1, \sigma_{\min} $ s.t.
$$\sigma_{\min} \in (0,\sigma_0],\ \theta_1 >0,\ 0<\eta_1\le\eta_2<1, \
0<\gamma_1 < 1 < \gamma_2, \ 0<C_{{\rm low}}<C_{{\rm up}}
$$
An integer $j_{\max}\ll n$.

\State Compute $f(x_0), g_0 = \nabla f(x_0), H_0 = \nabla^2 f(x_0)$, and set $k=0$, $\mathtt{refresh}=1$ and $V_0=[\, ]$.

 \State{\bf Step 1: Test for termination.} If $\|g_k\| \le \epsilon$, terminate.

 \State{\bf Step 2: Step computation.}
\State \hskip 0.2in Invoke Algorithm~\ref{projected_secular_algo} providing the scalar $\widehat \lambda_k$, the reduced step $\widehat s_k$, 
the basis 
\Statex \hskip 0.2in $W_{k}$ and  $V_{k+1}$ \label{call2}
\State \hskip 0.2in   {\bf if} 
$W_{k} \widehat s_k$ satisfies  \req{eq:dec2} \label{acc}
\hfill  
\State\hskip 0.4in     set $s_k = W_{k} \widehat s_k$ and {\bf goto} Step 3 \label{psub} 
\State \hskip 0.2in {\bf end if}
\State \hskip 0.2in   Compute $s_k = - ( H_k+\widehat \lambda_k I)^{-1} g_k$. \label{pnew}  \hfill   {\tt \% regularized Newton step} 
\State \hskip 0.2in     {\bf if} {$s_k^T(H_k+\widehat \lambda_k I) s_k \le 0 $ or {$\|s_k\|/\|\widehat s_k\|<C_{{\rm low}}$ or $\|s_k\|/\|\widehat s_k\|>C_{{\rm up}}$}}  \label{pd} 
\State \hskip 0.4in   {\bf if} {$\mathtt{refresh}$}\label{jmax}
\State \hskip 0.6in   \mbox{Find $ \lambda_k$ and
$s_k = -(H_k+ \lambda_k I)^{-1} g_k$ such that \eqref{eq:dec1}-\eqref{eq:dec2} hold by} 
\Statex \hskip 0.6in applying the secant method to
$\phi_R( \lambda_k; g_k, H_k, \sigma_k)=0$    
\label{secant} 
\State \hskip 0.6in   {\bf Goto} Step 3
\State \hskip 0.4in   {\bf else}
\State \hskip 0.6in     Define $x_{k+1}= x_k$, $\sigma_{k+1}= \sigma_k$, set $\mathtt{refresh}=1$, $k=k+1$\label{rejpd}
\Statex{\hskip 0.6in {\bf Goto} Step 1}
\hfill  {\tt \% unsuccessful iteration, step rejection}
\State \hskip 0.4in {\bf end if}

\State \hskip 0.2in {\bf end if}

\State {\bf Step 3:} Compute 
$$
\rho_k = \frac{f(x_k) - f(x_k+ s_k)}{T_k(0)-T_k(s_k)}.
$$
\State   {\bf Step 4:} %:  Successful iteration}
\If{$\rho_k \ge \eta_1$} \hfill  {\tt \% Successful iteration}
\State
set $x_{k+1}= x_k +  s_k$, compute $g_{k+1} = \nabla f(x_{k+1})$, $H_{k+1} = \nabla^2 f(x_{k+1})$
\Else \hfill    {\tt \% unsuccessful iteration, step rejection}
\State set $x_{k+1}= x_k$ \label{rejrho}
\EndIf

\State {\bf Step 5: Regularization parameter update.} Compute
   \begin{equation}\label{sig2bis}
   \sigma_{k+1} \in \left \{
   \begin{array}{lll}
    [\max\{\sigma_{\min}, \gamma_1 \sigma_k\}, \sigma_k  ]                   & \mbox{ if } \rho_k\geq \eta_2
      &\\ 
   \sigma_k &\mbox{ if }\rho_k\in [\eta_1,\eta_2)
      & \\
     \gamma_2 \sigma_k & \mbox{ otherwise }
      &  
   \end{array}
   \right .
    \end{equation}
   \State Increment $k$ by one, set $\mathtt{refresh}=0$ and {\bf goto} Step 1.

\end{algorithmic}
\end{algorithm}

%%%%%%%%%%
\begin{algorithm}[tb]
\caption{Minimization of $m_k$ over the low-dimensional subspace}\label{projected_secular_algo}
\begin{algorithmic}[1]
\Require The matrix $H_k$; the vector $g_k$; the matrix $V_{k}\in\mathbb{R}^{n\times d_k}$; the accuracy threshold $\theta_1>0$; the parameter $\sigma_k$;
an integer $j_{\max} \ll n$; {\tt refresh}.

\State \hskip 0.2in {\bf if} $\mathtt{refresh}$ \hskip 1.7in {\tt\% Generate new proj space}

\State \hskip 0.4in Set $V_{k+1}=[\, ]$ 

\State \hskip 0.4in {\bf for} $j=1,\ldots,j_{\max}-1$

\State \hskip 0.6in Set\footnotemark~ $W^{(j)}={\rm orth}([V_{k+1},g_k])$ \label{line_orth_W}
\State \hskip 0.6in Compute projections $g^{(j)} = (W^{(j)})^T g_k$, $\quad H^{(j)} = (W^{(j)})^T H_k W^{(j)}$
\State \hskip 0.6in Find $\widehat \lambda$ s.t. \label{froot}
{$\phi_R(\widehat \lambda; g^{(j)}, H^{(j)}, \sigma_k) =0,$ i.e.}
 \begin{equation*}
  \widehat \lambda=\sigma_k \|(H^{(j)}+\widehat \lambda I)^{-1} g^{(j)}\|
 \end{equation*}
 \State \hskip 0.6in Set
 $\widehat s = -(H^{(j)}+\widehat \lambda I)^{-1} g^{(j)}$
 \State \hskip 0.6in {\bf if} $\|\nabla m_k(W^{(j)}\widehat s) \| \le \frac{1}{2}\theta_1 \|\widehat s\|^2$
 \State \hskip 0.8in Set 
 $\widehat \lambda_k=\widehat \lambda$, $\widehat s_k=\widehat s$,  {$W_k =W^{(j)}$}  and {\bf return}
 \State \hskip 0.6in {\bf end if}
\State \hskip 0.6in Expand $V_{k+1}$ with new { basis vector}~\label{alg:expand_krylov}
\State \hskip 0.4in {\bf end for}
\State \hskip 0.4in Set 
 $\widehat \lambda_k=\widehat \lambda$, $\widehat s_k=\widehat s$,  {$W_k =W^{(j)}$}  and {\bf return} 
\State \hskip 0.2in {\bf else}
\hskip 2.2in {\tt\% Project onto the old space}
\State \hskip 0.4in Set $W^{(\widehat \jmath)}= {\rm orth}([V_{k}, g_k])$, where $\widehat \jmath=d_k+1$ %
\State \hskip 0.4in Compute projections $g^{(\widehat \jmath)} = (W^{(\widehat \jmath)})^T g_k$, $\quad H^{(\widehat \jmath)} = (W^{(\widehat \jmath)})^T H_k \REV{W^{(\widehat \jmath)}}$ 
\State \hskip 0.4in Find $\widehat \lambda$ s.t. \label{froot2}
{$\phi_R(\widehat \lambda; g^{(\widehat \jmath)}, H^{(\widehat \jmath)}, \sigma_k) =0,$ i.e.}
 \begin{equation*}
  \widehat \lambda=\sigma_k \|(H^{(\widehat \jmath)}+\widehat \lambda I)^{-1} g^{(\widehat \jmath)}\|
 \end{equation*}
 \State \hskip 0.4in Set
 $\widehat s_k = -(H^{(\widehat \jmath)}+\widehat \lambda I)^{-1} g^{(\widehat \jmath)}$, $\quad \widehat \lambda_k=\widehat \lambda$
 \State \hskip 0.4in Set $V_{k+1}=V_{k}$ and {$W_k= W^{(\widehat \jmath)}$}
 \REV{and {\bf return} }

\State \hskip 0.2in {\bf end if}

\end{algorithmic}
\end{algorithm}
The overall \frozen\ procedure is described in Algorithm~\ref{fAR2_refresh_algo} and Algorithm~\ref{projected_secular_algo}.
The outcomes of Algorithm~\ref{projected_secular_algo} are the regularization parameter  $\widehat \lambda_k$,  the step $\widehat s_k$, the matrix $W_k$ storing  the basis of the  subspace used to carry out 
the minimization process needed to generate $\widehat\lambda_k$ and $\widehat s_k$.
 The matrix $W_k$ is such that $\text{Range}(W_k)=\text{Range}([g_k,V_{k+1}])$. The columns of the matrix $V_{k+1}$ represent an orthonormal basis of the current subspace $\mathcal{K}$ and they are computed  from scratch whenever   the input parameter {\tt refresh} is equal to one. The minimization is 
carried out over the  augmented subspace spanned by $W_k$
to exactly project the current gradient $g_k$. 
The input parameter $j_{\max}\ll n$ rules the maximum allowed dimension of the subspace spanned by $V_{k+1}$. 

The reduced step $\widehat s_k$   
is the exact global minimizer of $m_k(s)$ onto the subspace
generated by  $W_{k}$. Moreover, the regularization parameter  has the following form:
\begin{equation} \label{lambdak}
\widehat \lambda_k =  \sigma_k \| \widehat s_k\|.
\end{equation}
In Line~\ref{call2} of Algorithm~\ref{fAR2_refresh_algo} 
 the minimizer $\widehat s_k$ of the cubic model onto the current subspace and the regularization parameter $\widehat \lambda_k$  are computed by invoking Algorithm~\ref{projected_secular_algo}. If $W_k\widehat s_k$ satisfies~\req{eq:dec2}, we set $s_k=W_k\widehat s_k$ and proceed with Step 3 while the current $\mathcal{K}$ is kept frozen. Otherwise, in Line~\ref{pnew}, the regularized Newton step $s_k=-(H_k+\widehat \lambda_kI)^{-1}g_k$ is computed and used in Step 3 whenever the
conditions \eqref{posedf_dir} and \eqref{bound_p} hold.
If one of these conditions is not satisfied, $s_k$ is rejected. In this scenario the current subspace is refreshed if it comes from previous iterations and the subspace minimization strategy is attempted once again. Otherwise, namely if $\mathcal{K}$ has just been refreshed, a new step $s_k$ 
satisfying \eqref{eq:dec1}-\eqref{eq:dec2} is computed by approximately solving the secular equation $\phi_R(\lambda;g_k,H_k,\sigma_k)=0$. The latter equation is solved (see \eqref{eq:sec}) by the secant method  (\cite{grt10} and \cite[Chapter 9]{book_compl}) equipped with suitable stopping criteria; see \cite[Theorem 9.3.2]{book_compl}.
The flag {\tt refresh} rules the computation of
a new subspace.  A new subspace is computed whenever {\tt refresh} is 1. The flag
{\tt refresh}
 is set to 1  at the first iteration and   whenever  the following  situation occurs
 \vskip 0.1in
 \begin{description}[align=left,style=multiline,labelindent=0.7cm]
 \item[$\clubsuit$] {\tt refresh} is currently set to $0$, the minimization of the model in the current subspace does not satisfy \eqref{eq:dec2}, and either \eqref{posedf_dir} or \eqref{bound_p}  does not hold. 
 \end{description} 
 \vskip 0.1in
One iteration is declared unsuccessful when $\clubsuit$ occurs or, as in classic adaptive methods, when the ratio $\rho_k $ in \eqref{eq_rho} is smaller than the input parameter $\eta_1\in (0,1)$. 
Note that, in case of $\clubsuit$, 
 the parameter $\sigma_k$ is left unchanged and a new subspace is computed at the subsequent iteration. We refrain from increasing $\sigma_k$ in this case as the failure is not ascribed to an unsatisfactory model but rather to a poor subspace. On the contrary, in case the unsuccessful iteration is due to $\rho_k<\eta_1$, the parameter $\sigma_k$ is increased.
The flag {\tt refresh} is set to 0 in Step 5 as this step is executed only if either \eqref{eq:dec2} holds or both \eqref{posedf_dir} and \eqref{bound_p} are satisfied. In this case we keep $\mathcal{K}$ frozen.

We observe that when the condition in Line \ref{acc} of Algorithm \ref{projected_secular_algo} is met, that is when the subspace minimization provides a reduced step $\widehat s_k$ such that
$W_{k}\widehat s_k$ satisfies~\req{eq:dec2}, then the computation of the regularized Newton step and the consequent 
factorization of the shifted Hessian matrix in the full space are not needed.

\footnotetext{Here and in the following, {\rm orth} defines a procedure that updates a matrix having orthonormal columns by adding a new column and orthogonalizing it with respect to the previous ones.}

We would like to point out the different nature of the step $s_k$ in the successful iterations. If the step is computed in Line \ref{psub} of Algorithm \ref{fAR2_refresh_algo}, 
then it is such that \req{eq:dec2} holds 
and $\|s_k\|=\|\widehat s_k\|$ since $W_k$ has orthonormal columns. 
On the other hand, if it is  computed in Line \ref{pnew} of Algorithm \ref{fAR2_refresh_algo}, 
it ensures a decrease of the quadratic regularized model
\begin{equation}\label{defmQ}
 m^Q_k(s)\eqdef T_k(s) + \frac{1}{2}\widehat \lambda_k \|s\|^2,
 \end{equation}
with $\widehat \lambda_k =  \sigma_k \|\widehat s_k\|,$
as the following inequality holds:
\begin{equation}\label{ineqmq}
m^Q_k(s_k) = T_k(s_k) + \frac{1}{2} \sigma_k \|\widehat s_k\| \|s_k\|^2
< m^Q_k(0).
\end{equation}
Indeed,
{\begin{equation}\label{ineqmq2}
m^Q_k(s_k) = f_k + s_k^Tg_k +\frac{1}{2} s_k^T(H_k + \hat \lambda_k I) s_k,
\end{equation}
and by \eqref{posedf_dir}   and 
$s_k^T(H_k + \hat \lambda_k I)s_k = -s_k^T g_k$, we get 
\begin{equation}\label{ineqmq3}
m^Q_k(s_k) - m^Q_k(0) =  s_k^Tg_k -\frac{1}{2} s_k^T g_k = \frac{1}{2} s_k^T g_k <0.
\end{equation}
Note that, in case $H_k+\widehat \lambda_k I$ is  positive definite, the step $s_k$  is the exact global minimizer of the quadratic regularized model $m^Q_k(s )$. }

When $s_k$ is computed in  Line \ref{secant} of Algorithm \ref{fAR2_refresh_algo} it is still an approximate  minimizer of the cubic model satisfying \eqref{eq:dec1}-\eqref{eq:dec2}.

We finally stress that condition \eqref{bound_p} holds whenever there exist
$0<\theta_{\min}<\theta_{\max}$ such that $\lambda_i(H_k+\widehat\lambda_k I)^{-2}\in [\theta_{\min},\theta_{\max}]$, for $i=1,\ldots,n$.
Indeed, we have
$$\|s_k\|^2=g_k^T(H_k+\widehat\lambda_k I)^{-2}g_k,$$
so that
$$\theta_{\min}\|g_k\|^2\leq\|s_k\|^2\leq
\theta_{\max}\|g_k\|^2.
$$
Similarly, since $g_k$ is always included in  the subspace employed to compute $\widehat s_k$, it holds
$$\widehat \theta_{\min}\|g_k\|^2\leq\|\widehat s_k\|^2\leq
\widehat \theta_{\max}\|g_k\|^2,
$$
where $\widehat\theta_{\min}$ and $\widehat\theta_{\max}$ denote the smallest and largest eigenvalues of $(\widehat H_k+\widehat\lambda_k I)^{-2}$, respectively. Notice that $\widehat\theta_{\min}$ and $\widehat\theta_{\max}$ are guaranteed to be strictly positive as $\widehat\lambda_k$ solves the reduced secular equation so that $\widehat H_k+\widehat\lambda_kI$ is certainly positive definite. 
Therefore,
$$\|g_k\|^2\leq\frac{\|\widehat s_k\|^2}{\widehat \theta_{\min}}, \qquad \frac{\|\widehat s_k\|^2}{\widehat \theta_{\max}}\leq\|g_k\|^2.
$$
By putting everything together, we have
$$\frac{\theta_{\min}}{\widehat \theta_{\max}} \|\widehat s_k\|^2\leq\|s_k\|^2\leq \frac{\theta_{\max}}{\widehat \theta_{\min}}
\|\widehat s_k\|^2.
$$

\section{Complexity analysis  of finding first-order optimality points}\label{Complexity}
Given $\epsilon>0$ we provide an upper bound on the number of iterations needed by \frozen\ to compute an $\epsilon$-approximate first-order optimality point (see~\eqref{fomin}).
To this end, we make the following standard assumptions on the optimization problem~\req{eq:pb}.
\vskip 0.1in
\begin{description}[align=left,style=multiline]
\item[AS.1] $f(x)$ is twice continuously differentiable in $\IR^n$, and its gradient $\nabla f(x)$ and its Hessian $\nabla^2 f(x)$ are Lipschitz continuous on $\IR^n$ with Lipschitz constants $L_1$ and $L_2$, respectively.

\item[AS.2] $f(x)$ is bounded from below in $\IR^n$, that is there exists a constant $f_{\rm low}$ such that $f(x) \ge f_{\rm low}$ for all $x \in \IR^n$.
\end{description}

\vskip 0.1in
We recall that, given $x_k, s_k \in \IR^n$ and by letting $T_k(s_k)$ be the second-order Taylor approximation of $f(x_k+s_k)$ around $x_k$, Assumption \textbf{AS.1} implies the following 
\begin{equation}\label{tay}
 |f(x_k+s_k)-T_k(s_k)| \le \frac{L_2}{6}\|s_k\|^3,
\end{equation}
 see, e.g., \cite[Corollary A.8.4]{book_compl}.

Let us denote by ${\cal I}_{RN}$ the set of indexes of iterations such that the \REV{tentative} step $s_k$ %used in Step 3
has been computed in Line \ref{pnew} of Algorithm \ref{fAR2_refresh_algo}, i.e., it is a regularized Newton step. 
We observe that at
any successful iteration $k\not \in {\cal I}_{RN}$ the step $s_k$ is the step   used in the classic  adaptive regularization method\footnote{It is an approximate minimizer satisfying \eqref{eq:dec1}-\eqref{eq:dec2}  both when it is computed in Line 6 and  in Line 11.} employing the  second order model. Then, we can rely on the  theoretical results given in 
\cite[Section 3.3]{book_compl}  for these iterations and prove analogous results for iterations in ${\cal I}_{RN}$. Our analysis  follows  the classic path for proving worst-case complexity of adaptive regularized methods despite a regularized Newton step is used.
We also observe that condition~\eqref{bound_p} holds at any successful iteration  $k\in {\cal I}_{RN}$ and that, as already noticed,  $\sigma_k$ is left unchanged at any unsuccessful iteration in 
$ {\cal I}_{RN}$ where \eqref{bound_p} does not hold.

We are now going to 
establish a key  lower bound on the Taylor-series model
decrease. In case $k\not \in {\cal I}_{RN}$, the classic lower-bound of AR2 methods holds. Remarkably, the lower bound depends on $\|s_k\|^3$ also for $k\in {\cal I}_{RN}$ provided that \eqref{bound_p} holds.
This lower bound is crucial to prove the optimal complexity of our procedure and it holds as $\hat \lambda_k$ is the minimizer of the cubic model in the low-dimensional subspace yielding $\hat \lambda_k=\sigma_k\|\hat s_k\|$.

\begin{lemma}{(\bf Decrease in the Taylor-series model)}
 At every successful  iteration $k$ of the \frozen\ algorithm, it holds that 
\begin{equation}\label{eq:diffT0}
 T_k(0) - T_k( s_k) > \frac{1}{2} \sigma_k \|\widehat s_k\| \|s_k\|^2,\quad\quad\quad  \mbox{for }  k \in {\cal I}_{RN},
\end{equation}
and 
\begin{equation}\label{eq:diffT1}
 T_k(0) - T_k( s_k) > \frac{1}{3} \sigma_k  \|s_k\|^3,\quad\quad\quad  \mbox{for }  k \notin {\cal I}_{RN}.
\end{equation}

\end{lemma}
\begin{proof} Let us consider the case $ k \in {\cal I}_{RN}$. 
Inequality \eqref{ineqmq} yields 
$$
 T_k(0) - T_k( s_k) = m^Q_k (0) - m^Q_k(s_k) + \frac{1}{2} \sigma_k \|\widehat s_k\| \|s_k\|^2 \ge\frac{1}{2} \sigma_k \|\widehat s_k\| \|s_k\|^2.
$$
In case  $ k \not \in {\cal I}_{RN}$, the result follows from \cite[Lemma~3.3.1]{book_compl}.
\end{proof}

\begin{lemma}{(\bf Upper bound on the regularization parameter)}
Suppose that Assumption {\rm\textbf{AS.1}} holds.
At each  iteration  of the \frozen\ algorithm  it holds
\begin{equation}\label{eq:bsigma}
\sigma_k \le  \sigma_{\max} \eqdef \max \left\{ \sigma_0, \gamma_2 \frac{ L_2C_{{\rm up}}}{3(1-\eta_1)}, \gamma_2 \frac{ L_2}{2(1-\eta_1)} \right \},
\end{equation} 
where $\gamma_2$ and $\eta_1$ are given in the algorithm. 
\end{lemma}
Let us denote by ${\cal I}_{RN}$ the set of indexes of iterations such that the step $s_k$ used in Step 3 has been computed in Line \ref{pnew} of Algorithm \ref{fAR2_refresh_algo}, i.e., it is a regularized Newton step. 
We observe that at
any successful iteration $k\not \in {\cal I}_{RN}$ the step $s_k$ is the step   used in the classic  adaptive regularization method

\begin{proof}
\REV{For any $ k \not \in {\cal I}_{RN}$, we recall that $s_k$ is the step used in the classic adaptive regularization method.
Hence, by exploiting
} (\ref{eq:diffT1})  and proceeding as in   \cite[Lemma 3.3.2]{book_compl} we can show that 
the iteration is successful whenever  
$$
\sigma_k\ge   \frac{ L_2}{2(1-\eta_1)}.
$$
In case  $ k \in {\cal I}_{RN}$, \REV{since the step $s_k$  is the regularized Newton step},   
the following two situations can occur:
\vskip 0.05in
\begin{description}[align=left,style=multiline,labelindent=0.7cm]
     \item[a)]  Both conditions \eqref{posedf_dir} and \eqref{bound_p} hold.
     \item[b)] Either  \eqref{posedf_dir} or \eqref{bound_p} does not hold. The iteration is declared unsuccessful in Line \ref{rejpd} of Algorithm \ref{fAR2_refresh_algo}.  {\tt refresh} is set to $1$ and $\sigma_k$ is left unchanged.
\end{description}
\vskip 0.05in
Let us consider case \textbf{a)}. The ratio $\rho_k$ is evaluated in Step 5 and 
from its definition  we have
\begin{eqnarray*}
 |\rho_k -1 | &=& \frac{|f(x_k) - f(x_k+ s_k) - T_k(0)+T_k(s_k)|}{T_k(0)-T_k(s_k)} \\
 &=& \frac{| f(x_k+ s_k) - T_k(s_k)|}{T_k(0)-T_k(s_k)} \\
  & \le & \frac{ L_2 \|s_k\|^3}{ 3\sigma_k \|\widehat s_k\|  \|s_k\|^2  } \\
  & = & \frac{ L_2 \|s_k\|}{ 3 \sigma_k \|\widehat s_k\|},
\end{eqnarray*}
where we used the inequalities (\ref{tay}) and (\ref{eq:diffT0}).
Therefore, by using  (\ref{bound_p}), we can conclude that  if  $k  \in {\cal I}_{RN}$ and case \textbf{a)} occurs, then   
$$
\sigma_k\ge   \frac{ L_2C_{{\rm up}}}{3(1-\eta_1)},
$$
which implies $\rho_k>\eta_1$, i.e.,    the iteration is successful.
Therefore, thanks to the mechanism of the algorithm we have $$
\sigma_k\le \gamma_2 \frac{ L_2C_{{\rm up}}}{3(1-\eta_1)},
$$ for any $k  \in {\cal I}_{RN}$ such that  case \textbf{a)} holds. 

Finally,  unsuccessful iterations $k  \in {\cal I}_{RN}$ such that  case \textbf{b)} holds  do not yield an increase in $\sigma_k $.

Therefore, $\sigma_k$ cannot be larger than $\sigma_{\max}$,
and the bound in (\ref{eq:bsigma}) follows.
\end{proof}

\REV{In the following lemma we prove the lower bound \eqref{eq:stepgrad} on the norm of the step $\widehat s_k$ in case  $ k \in {\cal I}_{RN}$. To this end we exploit the fact the regularized Newton step  $s_k$ 
provides $\| H_k s_k+g_k\|$ of the order of $\|\widehat s_k\|^2$, thanks to the form of 
regularization term $\widehat \lambda_k$ given in  \eqref{lambdak}.
The lower  bound \eqref{eq:stepgrad}, along with \eqref{eq:diffT0}-\eqref{eq:diffT1}, paves the way for
proving the optimal complexity result.}

 \begin{lemma} \label{lemma_decrease}
Suppose that Assumption {\rm\textbf{AS.1}} holds. Then,
at   any successful iteration $k$  of the \frozen\ algorithm, if  $ k \in {\cal I}_{RN}$  then
\begin{equation} \label{eq:stepgrad}
 \|\widehat s_k\|^2\ge \frac {2}{C_{\rm up}(L_1 C_{\rm up}+2\sigma_{\max})} \|g_{k+1}\|.
\end{equation}
Moreover, at any  successful iteration, the following inequality holds
\begin{equation}\label{decrease}
f(x_k)-f(x_{k+1})\ge \eta_1\sigma_{\min}\min\left \{\frac{1}{2} C_{\rm low}^2 \kappa_0,   \frac{1}{3} \kappa_1\right\} \, \|g_{k+1}\|^{3/2},
\end{equation}
where  $\kappa_0=\left( \frac{2}{C_{\rm up} (L_1 C_{\rm up}+2\sigma_{\max})}\right)^{3/2}$,  $\kappa_1= \left( \frac{2}{L_1+\theta_1+\sigma_{\max}}\right)^{3/2}$.
\end{lemma}

\begin{proof}
Let us consider the case $ k  \in {\cal I}_{RN}$. By  using the definition of $s_k$ in
Line~\ref{pnew} of Algorithm \ref{fAR2_refresh_algo},~\eqref{bound_p}, \eqref{eq:bsigma}, and the form of $\hat \lambda_k$ given in \eqref{lambdak},  it follows that
\begin{eqnarray} \label{bound_p1}
\|g_{k+1}\|& \le &\|g_{k+1}-(H_k s_k+g_k)\|+\| H_k s_k+g_k\|  \notag\\
&\le& \frac{L_1}{2}   \|s_k\|^2+ \sigma_k \|\widehat s_k\|\|s_k\|\notag\\
&\le& C_{\rm up} \left (\frac{L_1}{2} C_{\rm up}+\sigma_{\max} \right)   \|\widehat s_k\|^2.
\end{eqnarray}
Then, \eqref{eq:diffT0} yields
\begin{eqnarray*}
f(x_k)-f(x_{k+1})&\ge& \eta_1 (T_k(0)-T_k(s_k)) \\
&\ge& \frac{\eta_1}{2} \sigma_k \|\widehat s_k\|\|s_k\|^2 
\ge \frac{\eta_1}{2} \sigma_{\min} C_{\rm low}^2  \|\widehat s_k\|^3 \\
&\ge& \frac{\eta_1}{2} \sigma_{\min} C_{\rm low}^2 \kappa_0 \|g_{k+1}\|^{3/2},
\end{eqnarray*}
where $\kappa_0\eqdef \left( \frac{2}{C_{\rm up} (L_1 C_{\rm up}+2\sigma_{\max})}\right)^{3/2}$.

In case $ k \not \in {\cal I}_{RN}$, Lemma 3.3.3 in  \cite{book_compl} ensures that the step $s_k$ computed in Line~\ref{psub} of Algorithm \ref{fAR2_refresh_algo} satisfies:
$$
\|g_{k+1}\|\le\frac {L_1+\theta_1+\sigma_{\max}} { 2} \|s_k\|^2,
$$
and from  \eqref{eq:diffT1} it follows that
$$f(x_k)-f(x_{k+1})\ge \eta_1 (T_k(0)-T_k(s_k)) \ge \eta_1 \frac{1}{3} \sigma_k\|s_k\|^3 \ge  \eta_1  \frac{1}{3} \sigma_{\min} \kappa_1 \|g_{k+1}\|^{3/2},
$$
where $\kappa_1\eqdef \left( \frac{2}{L_1+\theta_1+\sigma_{\max}}\right)^{3/2}$.
\end{proof}

Let us denote by ${\cal S}_k$ the set of indexes of successful iterations up to iteration $k$ generated by the \frozen\ method detailed in Algorithms \ref{fAR2_refresh_algo}-\ref{projected_secular_algo}. Then, the following result allows us to bound the number of iterations in terms of the number of successful iterations provided that 
the regularization parameter is bounded from above.
 
\begin{lemma} \label{bound_it_S}  Suppose that  the sequence $\{x_k\}$ is generated by the \frozen\ algorithm. Assume  that $\sigma_k\le \sigma_{\max}$ for some 
$\sigma_{\max}>0$. Then, 
$$
k\le \left( |{\cal S}_k|\left(3/2+\frac{3/2|\log \gamma_1|}{\log \gamma_2}\right)+\frac{3/2}{\log \gamma_2}\log \left (\frac{\sigma_{\max}}{\sigma_0}\right)\right) .
$$
\end{lemma}
\begin{proof} 
Let ${\cal U}_{k,1}$ and ${\cal U}_{k,2}$
be the sets of indexes of the  iterations declared unsuccessful in Step 4 and in Line \ref{rejpd} 
up to iteration $k$ generated by the \frozen\ algorithm, respectively.
Observe that in case $j\in {\cal U}_{k,2}$, $\sigma_{j+1}=\sigma_j$. Then, by
proceeding as in  \cite[Lemma 2.4.1]{book_compl} we easily get 
that 
$$
|{\cal U}_{k,1}|+|{\cal S}_k|\le |{\cal S}_k|\left(1+\frac{|\log \gamma_1|}{\log \gamma_2}\right)+\frac{1}{\log \gamma_2}\log \left (\frac{\sigma_{\max}}{\sigma_0}\right).
$$
Since {\tt refresh} is set to $1$ at any iteration in  ${\cal U}_{k,2}$ and we can have an iteration in ${\cal U}_{k,2}$ only with {\tt refresh=0}, it follows that 
we can have at most  one iteration in ${\cal U}_{k,2}$ every two iterations in ${\cal U}_{k,1}\cup{\cal S}_k$. 
Then, the result
follows.
\end{proof}
\revdone{We can now resort to the standard “telescoping sum” argument to 
prove the following iteration complexity result 
showing that \frozen\ preserves the optimal complexity of \arc. }

\begin{theorem} \label{theorem_complexity}
Suppose that {\rm\textbf{AS.1}} and {\rm\textbf{AS.2}} hold. Given $\epsilon \in (0,1)$, there exists a constant $\kappa_p$ such that the \frozen\ algorithm 
needs at most
$$
\left \lceil \frac{ f(x_0)- f_{\rm low}}{\kappa_p}
 \epsilon^{-\frac{3}{2}}+1
\right \rceil,
$$
successful iterations 
and at most
$$
 \left( \left \lceil
 \frac{ f(x_0)- f_{\rm low}}{\kappa_p}
\ \epsilon^{-\frac{3}{2}} +1
\right \rceil
                 \left(3/2+\frac{3/2|\log\gamma_1|}{\log\gamma_2}\right)+
\frac{3/2}{\log\gamma_2}\log\left(\frac{\sigma_{\max}}{\sigma_0}\right)\right), 
$$
iterations to compute an iterate $x_k$ such that $\|g_k\|\le \epsilon$.
\end{theorem}

\begin{proof}At each successful iteration $k$  the algorithm guarantees
the decrease in ~\eqref{decrease}. 
Then,
\begin{equation}\label{decr-eps}
f(x_k)-f(x_{k+1})\ge
 \eta_1\sigma_{\min}\min \left \{\frac{1}{2} C_{\rm low}^2 \kappa_0,   \frac{1}{3} \kappa_1 \right \} \epsilon^{3/2},
\end{equation}
whenever  $\|g_{k+1}\|> \epsilon$.
Thus,  letting $\bar k$ such that $\bar k+1$ is the  first iteration where  $\|g_{\bar k+1}\|\le \epsilon$ and 
summing up from $k=0$ to $\bar k-1$, since $f(x_k)=f(x_{k+1})$ along unsuccessful iterations, and using  the telescopic sum   property, we obtain 
$$
 f(x_0)-f(x_{\bar k})
\geq \kappa_p \epsilon^{\frac{3}{2}}|{\cal S}_{\bar {k}-1}|,
\,
$$
where $\kappa_p=\eta_1 \sigma_{\min} \min \left \{\frac{1}{2}C_{\rm low}^2 \kappa_0,   \frac{1}{3} \kappa_1 \right \}$.

Since $f$ is bounded from below by $f_{\rm low}$, we conclude that
\begin{equation}\label{Sk1}
|{\cal S}_{\bar k} |
\leq \frac{f(x_0) - f_{\rm low}}{\kappa_p} \epsilon^{-\frac{3}{2}}+1.
\end{equation}
Lemma \ref{bound_it_S} then  yields the upper
bound on the total number of iterations.
\end{proof}

\begin{remark}
Given $\epsilon, \epsilon_H>0$, Algorithm~\ref{fAR2_refresh_algo}  can be suitably modified in order to compute an $(\epsilon,\epsilon_H)$-approximate second-order minimizer (see \eqref{fomin2}). We will refer to this variant of \frozen\  as Frozen AR2-Second Order (\frozenso) and the corresponding algorithms are detailed in Appendix~\ref{appendixA} (Algorithm~\ref{fAR2_refresh_algo_so} and Algorithm~\ref{projected_secular_algo2}),
together with a complexity analysis.
\end{remark}

\section{Choice of the approximation space}\label{sec:appr}
The performance of Algorithm~\ref{fAR2_refresh_algo} strongly depends on the chosen approximation space $\mathcal{K}_{d_k}=\text{Range}(V_k)$ 
used
for projecting the problem.
The construction of the approximation 
space and the definition of the related projected quantities are handled by Algorithm~\ref{projected_secular_algo}. The latter is given in full generality, regardless of the nature of $\text{Range}(V_k)$. However, some of its steps simplify according to the adopted subspace, as illustrated in the following.
Moreover, the performance of Algorithm~\ref{fAR2_refresh_algo} largely benefits from the inclusion of the current $g_k$ in the approximation space. 
In particular, the exact representation of $g_k$ 
 in the space is  crucial for obtaining the bounds on $\|s_k\|$ reported in~\eqref{bound_p}. 

In our framework, we define $V_k\in\mathbb{R}^{n\times d_k}$ as an orthonormal basis of the polynomial Krylov subspace generated  at iteration $d_k$ of \revdone{Algorithm~\ref{projected_secular_algo}}, by the current $H_k$ and $g_k$, namely
$$\mathcal K_{d_k}(H_k,g_k)=\text{span}\{g_k,H_kg_k\ldots,(H_k)^{d_k-1}g_k\}=\text{Range}(V_k).$$
By taking advantage of the symmetry of $H_k$, the basis $V_k$ can be constructed by the Lanczos method~\cite{Lanczos1952} which requires a single matrix-vector multiplication with $H_k$ per iteration along with a short-term orthonormalization procedure. Even though the entire basis is not required in this latter step, our \emph{freezing} strategy needs to store it to be able to keep using $V_k$ in later nonlinear iterations. This is in contrast with other contributions like, e.g.,~\cite{gltr}, where the basis is not stored in the first place but needs to be reconstructed by running a second Lanczos pass to retrieve the final solution.

We would like to stress that in case of polynomial Krylov subspaces, Line~\ref{line_orth_W} in Algorithm~\ref{projected_secular_algo} is redundant. Indeed, whenever a refresh is performed, $\text{span}(g_k)\subseteq \text{Range}(V_k)$ and $g_k$ is orthogonal to the following basis vectors by construction.

We also experimented with different approximation spaces. In particular, in section~\ref{rationalVSpolynomial} we report the results obtained by adopting the \emph{rational} Krylov subspace.
While the latter  oftentimes attains better numerical performance (see Tables~\ref{tab:risuopm}-\ref{tab:risuopm2}), rational Krylov subspaces turn out to be competitive on certain, particularly challenging, problems.

\section{Computational experiments}\label{Numerical experiments}
In this section we illustrate the  performance of our new \emph{freezing} paradigm for the computation of an approximate first-order optimality point.
We recall that  we are mainly interested in solving  optimization problems for which the solution of the regularized Newton linear system by direct methods is both feasible and efficient.  In our implementation of \frozen\ we employ direct methods for the solution of the regularized Newton system and compare \frozen\ against the state-of-the-art  implementation of \arc\ based on the secular equations and direct solvers. 

%tolto vecchio blu
\REV{In addition, in order to give more insight on the behaviour of \frozen\ against well-established procedure for unconstrained optimization problems,  we also  compare the results attained by \frozenp\ 
with the ones achieved by \galg, a state-of-the-art  implementation of \arc\ based on the Lanczos iteration \cite{ACO1}.}

The \frozen\ framework is designed for general approximation spaces; in our comparisons we adopt the polynomial Krylov subspace (\frozenp), while in section~\ref{rationalVSpolynomial} we also report some results obtained with the rational Krylov subspace as approximation space (\frozenr).

\subsection{Implementation details}
We implemented \frozen, whose details are given in Algorithms \ref{fAR2_refresh_algo}-\ref{projected_secular_algo},
in Matlab\footnote{MATLAB R2019b on a Intel Core i7-9700K CPU @ 3.60GHz x 8, 16 GB RAM, 64-bit.}~\cite{MATLAB} and we used the following standard values for the parameters
\begin{equation} \label{eq:param}
    \eta_1 = 0.1,\ \eta_2 = 0.8,\ \gamma_1=0.1,\ \gamma_2 = 2,\ \theta_1 = 0.1,\ \sigma_{\min} = 10^{-8},
\end{equation}
along with
$$\ C_{\rm low} = 10^{-20},\ C_{\rm up} = 10^{20}.
$$
\REV{In Step 5 of Algorithm  \ref{fAR2_refresh_algo}, in case $\rho_k\ge \eta_2$ we set $\sigma_{k+1}=\gamma_1 \sigma_k$ (cf. \req{sig2bis}).}
The maximum subspace dimension is set to $j_{\max} = 50$. 
Moreover, the solution of the secular equation projected onto $\text{Range}(W_k)$ (Lines \ref{froot} and \ref{froot2} of Algorithm \ref{projected_secular_algo}) is performed by using the Matlab function  {\tt fzero} providing $$[\min(-\lambda_{\min}(H^{(j)},0), \min(-\lambda_{\min}(H^{(j)}),0) +1000],$$  as starting interval for the bisection method. 
All the (shifted) linear systems involved in our procedure are solved by the Matlab backslash operator.

In the rare scenario where our \revdone{predictor-corrector}  scheme fails, namely Line~\ref{secant} of Algorithm \ref{fAR2_refresh_algo} is performed, 
 the step $s_k$ is computed by employing the state-of-the-art solver {\sc RQS} included in the 
{\sc GALAHAD} optimization library (version 2.3)\footnote{Available from http://galahad.rl.ac.uk/galahad-www/.} based, once again, on direct solvers for linear systems.
We used the Matlab interface and default parameters except for the values  ${\tt control.stop\_normal}\ = \theta/\sigma_k$ in order to fulfil the condition (\ref{eq:dec2}).

We compare the performance of our new method against a classic 
implementation of Algorithm \ref{AR2_algo} that makes use of the parameter values in (\ref{eq:param}) 
and uses {\sc RQS} in Step 2 equipped with all the default parameters except for the value  ${\tt control.stop\_normal}\ = \theta$. We refer to the corresponding overall implementation as \galr\footnote{We have also performed experiments employing the parameter value ${\tt control.stop\_normal}\ = \theta_1/\sigma_k$, but  
the corresponding implementation of \galr\ was much less robust than using ${\tt control.stop\_normal}\ = \theta_1$. Therefore, we chose to report the results obtained with the latter value only.}.
\REV{Finally, we considered \galg\ that in Step 2 uses {\sc GLRT}  from {\sc GALAHAD}  based on the Lanczos iteration.
Again we used default parameters except for the value  ${\tt control.stop\_relative}\ = \theta_1$. }

All algorithms terminate when one of the following conditions holds: 
\begin{equation}\label{stopg}
\text{(i)} \qquad \|g_k\| \le \epsilon_r \|g_0\|,
\end{equation}
where $\epsilon_r>0$ depends on the problem test set\footnote{The stopping criterion (\ref{stopg})
corresponds to setting $\epsilon= \epsilon_r \|g_0\|$ in Algorithms
\ref{AR2_algo} and \ref{fAR2_refresh_algo}.}; (ii) 5000 iterations are performed; (iii) the computational time limit of 2 hours is reached.

\subsection{Test problems}
We considered the test problems in the {\sc OPM} testing set~\cite{OPM}, which contains a small collection of {\sc CUTE}st unconstrained and bound-constrained nonlinear
optimization problems. From all the problems listed in~\cite{OPM} we selected the unconstrained ones. For those problems allowing for a variable dimension $n$, we specify the adopted value of $n$ in Table \ref{tab:OPM} in Appendix \ref{sec:opm}.

As testing problems, we also considered binary classification problems. We suppose we have at our disposal a training set of pairs $\{(a_i, b_i)\}$ with $a_i \in \IR^n, b_i \in \{-1, 1\}$ (or $b_i \in \{0, 1\}$)  and $i = 1, ... ,N$, where $b_i$ denotes the correct sample classification. We perform both a logistic regression and a 
sigmoidal regression, yielding convex and non convex problems, respectively.

In the first case, we consider as
training objective function the logistic loss with $\ell_2$ regularization \cite{BKK_IMA,Bollapragada2019}, i.e.
$$
f(x) = \frac{1}{N}\sum_{i = 1}^N \log \left(1+e^{-b_i a_i^T x}\right) + \frac{1}{2N} \|x\|_2^2,
$$
with $b_i \in \{-1, 1\}$. In the second case the following sigmoid function and least-squares loss is employed
\cite{SIRTR, Roosta_Informs}
$$
f(x) = \frac{1}{N}\sum_{i = 1}^N \left (b_i -\frac{1}{1+e^{- a_i^T x}}\right)^2,
$$
with  $b_i \in \{0, 1\}$. 
We used the six data sets\footnote{A9A \cite{UCI}, CINA0 \cite{UCI}, GISETTE \cite{UCI}, MNIST \cite{mnist}, MUSH \cite{UCI}, REGERD0 \cite{libsvm}} reported in Table \ref{tab:bin} together with the dimension $n$ of the parameter vector $x\in \IR^n$. 

For OPM and binary classification problems, we ran the algorithms until (\ref{stopg}) holds  with $\epsilon_r = 10^{-6}$ and $\epsilon_r = 10^{-3}$, respectively.

\subsection{Numerical results}
We  compare the performance of  \frozenp\  with
the state-of-the-art factorization-based \arc\ method, labeled as \galr,
\REV{and with the \arc\ state-of-the-art implementation adopting the Lanczos iterative scheme for the subproblems, labelled as \galg. 
To provide sensible insight on the general numerical behaviour of the
proposed \frozen\ framework, in the comparison between \frozenp\ and \galr\ we focus} on the number of $n$-dimensional factorizations and nonlinear iterations performed by the two algorithms.

\subsubsection{Results for the OPM test set}\label{expe:OPM}

The performance profiles related to 
\REV{\frozenp\ and \galr, applied to}
the OPM problems\footnote{All the tested algorithms failed in solving the very ill-conditioned problems ARGLINB, ARGLINC, COSINE, GENHUMPS, PENALTY2, and SCOSINE. Therefore, these problems have been excluded from the testing set. Moreover, they all did not converge in 2 hours when trying to solve CHEBYQUAD.} listed in Table~\ref{tab:OPM} are reported in Figure  \ref{figure:fig2}. 
 We remind the reader that
a performance profile graph $p_A(\tau)$ of an algorithm $A$ at point $\tau$ shows the fraction of the test set
for which the algorithm is able to solve within a factor of $\tau$ of the best algorithm for the given measure \cite{dolanmore}. 
 These performance profiles take into account the number of $n$-dimensional factorizations and nonlinear iterations as performance measures.

%%%%%%%%%%%%%%%%%%%%%%%%%%%%%%%%%%%%%%%%%%%%%%%%%%%%%%%%%%%%%%%%%%%%%%%%%

\begin{figure}[t!b] \centering
		\begin{tikzpicture}
		\begin{axis}[width=.8 * \linewidth, 
		title = {Number of factorization performance profile},
		xlabel = {$\tau$}, 
		ylabel = {$p_A(\tau)$},
		height = .25 \textheight,
		legend pos = south east,
		legend cell align={left},
		xmax = 40,
		xmin = 1,
		ymax = 1.1,
		ymin = 0, ymax = 1.1
		]
\addplot+[thick, mark=none, blue] table {./pp_OPM_fact_RQS_rev_2curve.dat};		
\addplot+[ultra thick, mark=none,  dotted] table {./pp_OPM_fact_frozenPK_rev_2curve.dat};
\legend{\galr, \frozenp};
\end{axis}
\end{tikzpicture}\\
		\begin{tikzpicture}
		\begin{axis}[width=.8 * \linewidth, 
		title = {Number of nonlinear iteration performance profile},
		xlabel = {$\tau$}, 
		ylabel = {$p_A(\tau)$},
       height = .25 \textheight,
		legend pos = south east,
		legend cell align={left},
		xmax = 16,
		xmin = 1,
		ymax = 1.1,
		ymin = 0, ymax = 1.1
		]
\addplot+[thick, mark=none, blue] table {./pp_OPM_NLI_RQS_rev_2curve.dat};		
\addplot+[ultra thick, mark=none,  dotted] table {./pp_OPM_NLI_frozenPK_rev_2curve.dat};
\legend{\galr,\frozenp};
\end{axis}
\end{tikzpicture}		
  \caption{Section~\ref{expe:OPM}. Performance profiles of \galr\ and \frozenp\ for the OPM set. Top: Number of factorizations. Bottom: Number of nonlinear iterations.\label{figure:fig2}}
\end{figure}
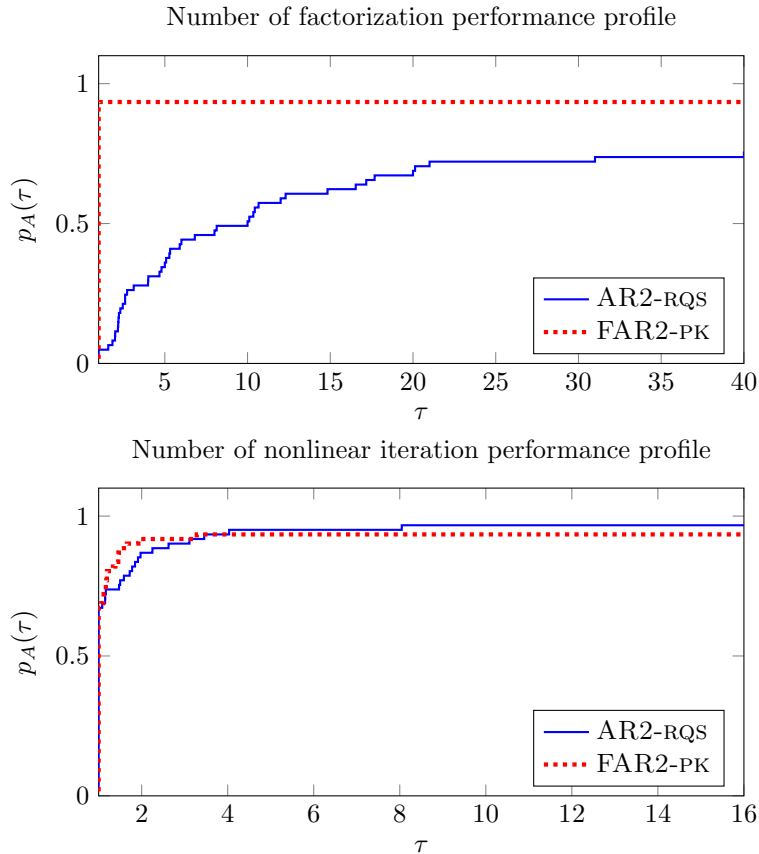

The profiles in Figure~\ref{figure:fig2} clearly show an outstanding advantage in using  our \frozenp\ approach in terms of total number of matrix factorizations with respect to \galr.  \revdone{Indeed, \frozenp\ is the most efficient routine in 94\%  of the runs
and \galr\ is within a factor 2 of \frozenp\  in only 11\%  of the runs.}
The behavior of the two routines in terms of number of nonlinear iterations is rather similar, while \galr\ is slightly more robust solving one problem more than \frozenp \footnote{The failure occurred because the maximum number of iterations was reached.}.

The first two sets of columns in Tables \ref{tab:risuopm} and \ref{tab:risuopm2} in Appendix~\ref{sec:opm} collect the detailed results used to generate the plots  in Figure~\ref{figure:fig2}. In particular, we report the number of nonlinear iterations ($\#$NLI) and the number of $n$-dimensional factorizations (\#fact) for the two solvers and for each problem. In addition, for the \frozen\ method the tables include also the number of nonlinear iterations where the space is refreshed (\#ref), the average dimension of the projected problems ($ave_K$), the number of times
the step $s_k$ is computed in Line \ref{psub} of 
Algorithm \ref{fAR2_refresh_algo} (\#sub) using a frozen subspace, and the number of times the solver resorts to the secant method in Line \ref{secant} of 
Algorithm \ref{fAR2_refresh_algo} (\#sec). These results provide a good understanding of the general behaviour of our algorithm. First, we observe that the basis refresh is invoked rarely, specifically in the 1.8\% of the total number of nonlinear iterations. Second, in the 3.5\% of the nonlinear iterations, the step 
$s_k$ is computed in Line \ref{psub}, that is condition (\ref{eq:dec2}) is satisfied avoiding the solution of the linear system in Line~\ref{pnew} thus 
saving a lot of computational efforts.  Third, the results clearly show that the secant method in Line~\ref{secant} is a mere safeguard step as it is
called 0.8\% of times by \frozenp.

Finally, we have compared \frozenp\ with \galg.  Considering the 67 OPM problems, for most runs  the performance of \frozenp\ and \galg\ is comparable.
 Nevertheless, we report in Table \ref{tab:gal} statistics of the subset of  problems where the dimension of the underlying polynomial Krylov subspace constructed within \galg\ is indeed the full dimension $n$ (or very close). 
 The table reports the number of nonlinear iterations (\#NLI), the maximum dimension of the generated Krylov subspace $max_K$ for \galg,  and the overall CPU time in seconds ($cpu*$)\footnote{We remark that we report the cpu time for illustrative purposes only, as both \galg\  and \frozenp\ use compiled functions.}. 

\REV{Even though constructing a basis for the full space is not an issue for \galg\ in terms of storage demand thanks to the employment of the Lanczos method for the basis construction, this does remarkably increase the overall computational effort. Indeed, at each iteration of the Lanczos process within \galg\ we have to solve a ``projected'' problem that is tridiagonal, yet its dimension is close to $n$. Repeating this operation for a number of iterations causes the different performance in terms of running time of 
 \frozenp\ and \galg\ in spite of the often similar number of nonlinear iterations they perform.
 These results show that when \galg\ struggles in the subspace minimization and an efficient preconditioner is not available, \frozenp\ results in a 
 valid alternative for the solution of~\eqref{eq:pb}.}
 
 \begin{table}[htbp]
  \centering
  \footnotesize
    \begin{tabular}{lrrrr}
          \toprule
          & \multicolumn{2}{c}{\galg}             & \multicolumn{2}{c}{\frozenp} \\
          \midrule
    Problem  & \multicolumn{1}{l}{\#NLI($max_K$)} &  \multicolumn{1}{l}{$cpu*$} &   \multicolumn{1}{l}{\#NLI} & \multicolumn{1}{l}{$cpu*$} \\
    \midrule
    CURLY10 & 30 ($n$) & 3.0  & 20 & 1.7 \\
    CURLY20 & 27 ($n$) & 2.8  & 24 & 1.9 \\
    CURLY30 & 29 ($n$) & 3.3  & 29 & 2.1 \\
    DIXON & 9 ($n$)  &  0.7   &  6 & 0.4 \\
    EIGENBLS & 750 (862) & 788&  739  & 373 \\
    MSQRTALS & 172 ($n$) & 195 & 256 & 142 \\
    MSQRTBLS & 332 ($n$) & 470 & 468 & 294 \\
    WMSQRTALS & 191 ($n$) & 253 & 292 & 153 \\
    WMSQRTBLS & 302 ($n$) & 493 & 352  & 191 \\
    \bottomrule
        \end{tabular}%
      \caption{Numerical results for selected OPM problems:
       number of nonlinear iterations (\#NLI), maximum dimension of the generated Krylov subspace $max_K$ for \galg, and the overall cpu time in seconds ($cpu*$).}
  \label{tab:gal}%
\end{table}%

% \end{document}

\subsubsection{Results for the binary classification problems}\label{expe:classification}
 Table~\ref{tab:bin} collects the results \REV{obtained solving the binary classification problems by  \frozenp\ and \galr,}  distinguishing between the logistic and the sigmoidal loss function cases for the six datasets we tested. In spite of the larger number of nonlinear iterations ($\#$NLI)  performed by \frozenp\ in a few cases, 
we remark that the overall number of factorizations always remains much in favour of \frozenp\ when compared to \galr. In particular, \frozenp\ is extremely efficient in terms of \#fact while keeping the dimension of the polynomial Krylov subspace rather low. 

\begin{table}[t!]
  \centering
\footnotesize
    \begin{tabular}{l|rr|rr}
    \toprule
    \multicolumn{1}{c|}{logistic loss} &        \multicolumn{2}{c|}{\# fact($ave_K$)} & \multicolumn{2}{c}{\#NLI} \\
    \midrule
    \multicolumn{1}{c|}{data set($n$)}  & \multicolumn{1}{c|}{\galr} & \multicolumn{1}{c|}{\frozenp} &  \multicolumn{1}{c|}{\galr} & \multicolumn{1}{c|}{\frozenp}   \\
    \midrule
     A9A (123)   & 16    & 3(2)            & 7   & 8  \\
   CINA0 (132)   & 36    & 1(26)          & 7    &  6  \\
   GISETTE (5000)  & 24    & 1(10)         & 7   & 7  \\
    MNIST (784)   & 22    & 4(4)           & 8     &  8  \\
    MUSH  (112)   & 20    & 0(3)           & 9  & 19    \\
    REGED0 (999)   & 8    & 7(50)          & 8       & 7  \\
    \midrule

    \midrule
    \multicolumn{1}{c|}{sigmoidal loss} &        \multicolumn{2}{c|}{\# fact ($ave_K$)} & \multicolumn{2}{c}{\#NLI} \\
    \midrule
    \multicolumn{1}{c|}{data set($n$)} &  \multicolumn{1}{c|}{\galr} & \multicolumn{1}{c|}{\frozenp}   &  \multicolumn{1}{c|}{\galr} & \multicolumn{1}{c|}{\frozenp}   \\
    \midrule
 A9A (123)   & 19    & 3(2.5)          & 8      & 21   \\
  CINA0 (132)   & 49    & 1(26)           & 7     &  6  \\
  GISETTE (5000)  & 58    & 4(7.6)            & 9      & 37    \\
   MNIST (784)   & 116   & 30(3)         & 16     & 39   \\
    MUSH  (112)   & 27    & 0(3)         & 10     &  33  \\
    REGED0 (999)   & 14    & 6(50)        & 7      &6  \\
    \bottomrule
    \end{tabular}%
      \caption{Section~\ref{expe:classification}. Results on binary classification problems: number of factorization (\# \revdone{fact}) and number of nonlinear iterations (\#NLI) for the \galr\ and the \frozenp\ methods.}     
  \label{tab:bin}%
  \end{table}%

\begin{figure}[tb]
\centering
\begin{tikzpicture}
\begin{axis}[width=. 49* \textwidth, 
		   title = {CINA0  data set with logistic loss},		   
	xtick=data,	
	legend style={at={(0.5,-0.3)},anchor=north},
	ylabel={\# factorizations},
	xlabel={nonlinear iterations},
	ymax = 12,  ytick={1,2,...,12},
		   ymin = 0,
	x tick label style=
		{rotate=45,anchor=east}]
\addplot+[ybar,mark=o ] plot coordinates
        {(1,5) (2,6)  (3,7) (4, 6) (5,5) (6,4) (7,3)};
                  \addplot+[ybar,mark=star] plot coordinates
        {(1,0) (2,1)  (3,0) (4, 0) (5,0) (6,0)};
            \legend{\galr,  \frozenp };
\end{axis}
 \end{tikzpicture}
 \begin{tikzpicture}
\begin{axis}[width=. 49* \textwidth, 
		   title = {CINA0  data set with sigmoid loss},		   
	 legend style={at={(0.5,-0.3)},anchor=north},
		xtick={1,2,...,9},
    ytick={1,2,...,10},
			ymax = 10,
		   ymin = 0,
	ylabel={\# factorizations},
	xlabel={nonlinear iterations},
	x tick label style=
		{rotate=45,anchor=east}]
\addplot+[ybar,mark=o  ] plot coordinates
        {(1,5) (2,6)  (3,7) (4, 8) (5,8) (6,8) (7,7)};
        \addplot+[ybar,mark=star ] plot coordinates
        {(1,0) (2,1)  (3,0) (4, 0) (5,0) (6,0) };
   \legend{\galr,  \frozenp };   
\end{axis}
 \end{tikzpicture}
 \caption{Section~\ref{expe:classification}. Number of factorizations per nonlinear iteration for the CINA0 data set.\label{fig:cina0}}
  \end{figure}

\begin{figure}[tb]
\centering
\begin{tikzpicture}
\begin{axis}[width=. 49* \textwidth, 
		  title = {MNIST data set with logistic loss},		   
		xtick={1,2,...,10},
  ytick={0,1,...,5},
			ymax = 5,
		   ymin = 0,
  legend style={at={(0.5,-0.3)},anchor=north},
 	ylabel={\# factorizations},
	xlabel={nonlinear iterations},
	x tick label style=
		{rotate=45,anchor=east}]
\addplot+[ybar,mark=o  ] plot coordinates
        {(1,2) (2,2)  (3,2) (4, 3) (5,3) (6,4) (7,3) (8,3)};
                 \addplot+[ybar,mark=star  ] plot coordinates
        {(1,0) (2,0)  (3,0) (4, 0) (5,1) (6,1) (7,1) (8,1)};
        \legend{\galr,  \frozenp };
\end{axis}
 \end{tikzpicture}
 \begin{tikzpicture}
\begin{axis}[width=. 49* \textwidth, 
			  title = {MNIST data set with sigmoid loss},		   
		xtick={1,8,...,57},
      legend style={at={(0.5,-0.3)},anchor=north},
    	ylabel={\# factorizations},
	xlabel={nonlinear iterations},
		ymax = 22,
		   ymin = 0,
	x tick label style=
		{rotate=45,anchor=east}]
\addplot+[ybar,mark=o  ] plot coordinates
        {(1,1) (2,2)  (3,2) (4, 3) (5,3) (6,4) (7,8) (8,8) (9,21) (10,14) (11,15) (12,8) (13,7) (14,10) (15,5) (16,5)};
    \addplot+[ybar,mark=star ] plot coordinates
        {(1,0) (2,0)  (3,0) (4, 0) (5,0) (6,1) (7,1) (8,1) (9,0) (10,0) (11, 1) (12,0) (13,1) (14,1) (15,1) (16,1)(17,1)(18,1)(19,1)(20,0)(21,1)(22,1)(23,1)(24,1)(25,1)(26,1)(27,1)(28,1)(29,1)(30,1)(31,1)(32,1)(33,1)(34,1)(35,1)(36,1)(37,1)(38,1)(39,1)};
      \legend{\galr,  \frozenp };
\end{axis}
 \end{tikzpicture}
 \caption{Section~\ref{expe:classification}. Number of factorizations per nonlinear iteration for the MNIST data set.}\label{fig:mnist}
 \end{figure}

Figures~\ref{fig:cina0}-\ref{fig:mnist} provide an easier visualization of the behaviour of the two routines we tested in terms of number of factorizations for two reference problems: CINA0 and MNIST. In all cases, we observe that \frozenp\ perform very few factorizations and the subspace-freezing strategy typically leads to a single factorization per nonlinear iteration or even no factorizations at all if the approximate solution $s_k=W_k\widehat s_k$ to the secular equation satisfies \req{eq:dec2} in Line \ref{psub} of Algorithm \ref{fAR2_refresh_algo}  (see Figures~\ref{fig:cina0}-\ref{fig:mnist}). This happens for several nonlinear iterations.

The sigmoidal case of MNIST in Figure \ref{fig:mnist} shows an unusual trend of \frozenp where many more nonlinear iterations are performed compared to what happens for \galr\ and the regularized Newton step is used frequently. This is probably the worst scenario for \frozen. Indeed, computing the regularized Newton step in many nonlinear iterations can deteriorate the gains coming from our freezing strategy as the number of linear systems to be solved may dramatically increase. Nevertheless, \frozen\ is still able to perform a number of factorizations that is much lower than the one needed by \galr; see Table~\ref{tab:bin}.

Finally, we observe that for strictly convex problems, as those obtained by using the logistic loss in classification problems, the refresh strategy is never activated as, in fact, $H_k+\lambda I$ is positive definite for any $\lambda \ge 0$
and condition (\ref{bound_p}) is always met in practice.

\subsubsection{On the use of a different approximation space: a numerical illustration}\label{rationalVSpolynomial}
As already mentioned, our new \frozen\, framework is rather flexible and, in principle, any approximation space can be adopted in our solution procedure. In this last experiment we explore the use of rational Krylov subspaces as approximation space.
To the best of our knowledge, rational Krylov subspaces have never been proposed in the context of nonlinear optimization prior this work. We thus report additional details about its construction and properties.

Given a suitable set of shifts $z_{d_{k}}=(\xi_1,\ldots,\xi_{d_{k}})\in\mathbb{R}^{d_{k}}$, the  rational Krylov subspace of dimension $d_k$ generated by the symmetric matrix $H_k$ and the vector $g_k$ is defined as follows\footnote{For convenience, we do not include $g_k$ in the space definition, following the typical setting in model order reduction.
Nonetheless, the vector $g_k$ is included in the projection phase, see Algorithm~\ref{projected_secular_algo}.}
\begin{equation}\label{def:ratkrylov}
{\cal K}_{d_{k}}(H_k,g_k,z_{d_{k}})=\text{span}\left\{%g_k,
(H_k+\xi_1 I)^{-1}g_k, \ldots,
\prod_{j=1}^{d_{k}}(H_k+\xi_j I)^{-1}g_k\right\}.  
\end{equation}
An orthonormal basis $V_k\in\mathbb{R}^{n\times d_{k}}$ of ${\cal K}_{d_{k}}$ can be constructed by means of the rational Arnoldi method~\cite{Ruhe94} or the rational Lanczos scheme~\cite{rationalLanczos}. Starting with $v_1=(H_k+\xi_1 I)^{-1}g_k$, normalized to have unit norm, the next basis vector is obtained by first computing
$\widetilde v_2=(H_k+\xi_2 I)^{-1}v_{1}$,
 orthogonalizing $\widetilde v_2$ with respect to the previous vectors, in this case only $v_1$, and finally normalizing to have unit norm. The procedure thus continues recursively.

One of the major computational costs per iteration of the rational Krylov subspace basis construction amounts to the solution of a shifted linear system with $H_k$, especially in case $H_k$ is not very sparse. Therefore, to limit computational costs the number of performed iterations $d_k\leq j_{\max}$ should remain small. 
To this end, the choice of the shifts is crucial. 
Quasi-optimal and adaptive procedures  have been proposed, where the shifts are computed automatically; see, e.g.,~\cite{Guettel2013,DruSim11_RKSM,Druskin2010} and the references therein. In case $H_k$ is semi-definite, as it is expected to occur for large enough $k$,  the nonzero shifts are chosen to have the same signature of $H_k$ so that $H_k + \xi_j I$ results in a positive definite matrix. On the other hand, if $H_k$ is slightly indefinite as it may happen at the beginning of the outer nonlinear procedure, the shifts can be selected to ensure that $H_k + \xi_j I$ is still nonsingular. 
However, we would like to remark that we never needed to implement
special strategies for the singular cases during our extensive experimental testing. In particular, although some instances of $H_k+\xi_j I$ turned out to be ill-conditioned, this aspect did not seem to severely affect the quality of the computed subspace.
The adaptive shift selection procedure described in~\cite{DruSim11_RKSM} has been adopted to obtain all the numerical results reported in this section.

Rational Krylov spaces are generally able to capture good spectral information on a given problem already for small dimensions, compared to classic polynomial Krylov subspaces, thus significantly reducing the original problem size. 
Moreover, we expect the rational Krylov subspace generated by $H_k$ and $g_k$ to still be a successful approximation space for successive nonlinear iterations, especially if the spectral properties of $H_k$ slowly change with $k$. On the other hand, the rational Krylov basis construction is rather expensive as it requires the solution of linear systems.

In Tables~\ref{tab:risuopm}-\ref{tab:risuopm2} we report the results obtained by solving all the OPM problems by \frozen\ adopting either the rational Krylov (RK) or the polynomial Krylov (PK) subspace as approximation space.
\revdone{
As it can be seen from these results, employing RK very often leads to an overall number of factorizations that is smaller than the one performed by adopting \galr\ but, almost always larger  than the one required employing PK}. We can thus confidently say that PK is often the right choice as approximation space in this regard. However, in case the overall number of outer iterations performed by adopting PK is significantly higher than employing RK,
solely looking at the  number of factorizations may give only a partial picture of the actual numerical performance of the whole procedure. This is especially the case for sparse problems where the linear system solution cost is proportional to the number of nonzeros of the coefficient matrix~\cite{ScottTuma2023}.
%%%%%%%%%%%%%%%%%%%%%%%%%%%%%%%%%%%%%%%%%%%%%%%%%%%%%%%%%%%%%%%%%%%%%%%%%%%%%%%%%

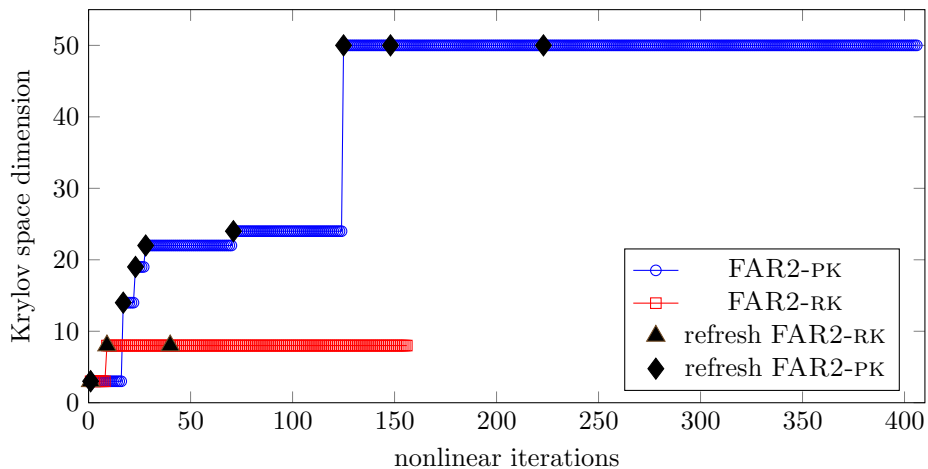
\begin{figure}[tb]
\centering
\begin{tikzpicture}
\begin{axis}[width= \textwidth, height=0.3\textheight,
		 legend pos=north west,
	ylabel={Krylov space dimension},
	xlabel={nonlinear iterations},
		   ymin = 0,
     xmin =0,
     xmax = 410,
  	ytick={0,10,...,60},	
	 legend pos= south east,]
	\foreach \j in {1} {
    \addplot[mark=o,mark options={fill=white}, blue] table[x index = 0, y index = \j] {./eigenbls_dimK_poly.dat};
		}
  \foreach \j in {1} {
		  \addplot+ [red,mark=square] table[x index = 0, y index = \j] {./eigenbls_dimK_rat.dat};
\addplot+[only marks, mark size=4, mark=triangle*,mark options={fill=black} ] plot coordinates
        {(1,3) (9,8)  (40,8)};
        \addplot+[only marks, mark size=4,mark=diamond*,mark options={fill=black} ] plot coordinates
        {(1,3) (17,14)  (23,19) (28,22) (71,24) (125,50) (148,50) (223,50) };
		}
        \legend{\frozenp, \frozenr,refresh \frozenr, refresh \frozenp };
\end{axis}
 \end{tikzpicture}
 \caption{Section~\ref{rationalVSpolynomial}. Dimension of the Krlylov subspace versus the nonlinear iterations for the EIGENBLS problem in OPM and subspace refresh: \frozenr\ and \frozenp.}
  \label{fig:eig2}
 \end{figure}

As a representative example, we report in Figure \ref{fig:eig2} the results obtained for the EIGENBLS problem from OPM with $\epsilon_r = 10^{-4}$ in (\ref{stopg}) for the sake of better readibility\footnote{We remind the reader that the results in Tables~\ref{tab:risuopm}-\ref{tab:risuopm2} are obtained by setting $\epsilon_r=10^{-6}$. This explains the different values reported in Figure~\ref{fig:eig2}.}:  for each nonlinear iteration the dimension of the generated Krylov subspace is reported.  We also highlight those iterations where a refresh of the Krylov subspace takes place. We can observe that the dimension of the PK subspace reaches the maximum dimension allowed ($j_{\max}=50$)
for large $k$. This is probably due to the large conditioning of the Hessian matrix which becomes of the order of $10^7$ for $k>100$.
Working with a larger subspace implies an increase in the cost of computing the projected step $\widehat s_k$ in Algorithm~\ref{projected_secular_algo}.
On the contrary, the dimension of the RK subspace stays always below 8
making the cost of computing $\widehat s_k$ negligible. Another evident drawback of employing PK in this example is the much larger number of nonlinear iterations we need to perform if compared to the case where RK is adopted. This leads to an increase in the cost of basically every step of our scheme: The projection phase (e.g., the explicit construction of $H^{(j)}=(W^{(j)})^TH_kW^{(j)}$), the regularized Newton step computation, and, more remarkably, the Hessian matrix $H_k$ and the gradient $g_k$ evaluation; \frozenp\ carries out all these steps hundreds of times more than \frozenr\, in this example. In particular, for both \frozenp\ and \frozenr, constructing $H_k$ and $g_k$ amounts to approximately the 85\% of the overall computational time. 
As a consequence, \frozenr\ turns out to be faster than \frozenp\ in spite of the more expensive basis construction of the former. In particular, in this example, the overall running time of \frozenr\ and \frozenp\ is
86.2 and  233.9 seconds, respectively\footnote{For illustrative purposes only, we report that \galr\  takes 517.8 seconds for the EIGENBLS problem.}.

\section{Discussion and conclusions}\label{Conclusions}
For the practical use of the adaptive regularization method, the efficient minimization of the cubic regularized model
 is a crucial issue. We have addressed this bottleneck by proposing a new two-step procedure, the \frozen\ method. The main difference between \frozen\ and standard adaptive-regularization techniques lies in the minimization of
the involved model: in  \frozen\ this minimization is carried out over a possibly {frozen} subspace.  Under standard  conditions,
 we have proved that the optimal worst-case complexity of the original \arc\ method is preserved, regardless of the adopted subspace. 

Our extended experimental testing shows that the freezing paradigm is able to considerably reduce the overall number of matrix factorizations compared to classic procedures based on the solution of the full space secular equation. In terms of approximation space, we have explored the use of both rational and polynomial Krylov subspaces. While they both lead to a rather successful solution process, the use of \frozen\ combined with polynomial Krylov subspaces turns out to be extremely efficient in most cases, whereas the implementation employing rational Krylov subspaces seems 
a valid alternative especially for sparse problems, where the additional cost coming from linear systems solves involved in the basis construction remains limited.

We finally remark that the proposed approach could be easily adapted to trust-region frameworks with second order models; see, e.g. {\sc TR2} in \cite[Section 3.2]{book_compl}. Our reduced approach could be appealing whenever the solution of the trust-region subproblem beyond the Steihaug–Toint point is sought; see, e.g., \cite{gltr,sigma}. Indeed, in this case, if the current step lies on the trust-region boundary,  the following secular equation 
$$
\phi_T(\lambda; g_k, H_k, \Delta_k) \eqdef \|(H_k + \lambda I)^{-1}g_k\| - \Delta_k =0,
$$
has to be (approximately) solved; here $\Delta_k>0$ is the so-called trust-region radius. Therefore,  Algorithm \ref{projected_secular_algo} can be straightforwardly modified by using $\phi_T$ in place of $\phi_R$ in Lines \ref{froot} and \ref{froot2}.

\section*{Declarations}
\subsection*{Ethics approval and consent to participate}
Not applicable.
\subsection*{Consent for publication}
Not applicable.
\subsection*{Funding}
{\footnotesize
The work of all the authors was partially supported by INdAM-GNCS under the INdAM-GNCS project CUP\_E53C22001930001. 
The research of S. B. was  partially granted by the Italian Ministry of University and Research (MUR) through the PRIN 2022 ``Numerical Optimization with Adaptive Accuracy and Applications to Machine Learning'',  code: 2022N3ZNAX MUR D.D. financing decree n. 973 of 30th June 2023 (CUP B53D23012670006). 
The research of S. B. and M. P. was partially granted by PNRR - Missione 4 Istruzione e Ricerca - Componente C2 Investimento 1.1, Fondo per il Programma Nazionale di Ricerca e Progetti di Rilevante Interesse Nazionale (PRIN) funded by the European Commission under the NextGeneration EU programme, project ``Advanced optimization METhods for automated central veIn Sign detection in multiple sclerosis from magneTic resonAnce imaging (AMETISTA)'',  code: P2022J9SNP,
MUR D.D. financing decree n. 1379 of 1st September 2023 (CUP E53D23017980001).
The work of D. P. and V. S.
was partially supported by the European Union - NextGenerationEU under the National Recovery and Resilience Plan (PNRR) - Mission 4 Education and research
- Component 2 From research to business - Investment 1.1 Notice Prin 2022 - DD N. 104 of 2/2/2022,
entitled “Low-rank Structures and Numerical Methods in Matrix and Tensor Computations and their
Application”, code 20227PCCKZ – CUP J53D23003620006. 
}
\subsection*{Availability of data and materials}
The data that support the findings of this study are available from the corresponding author upon request.
\subsection*{Competing interests}
The authors declare that they have no competing interests.
\subsection*{Authors' contributions}
All authors contributed equally to the writing of this article. All authors reviewed the manuscript.

%\subsection*{Acknowledgments}

% 
\appendix
\vskip 5pt
\noindent
\section{The Frozen AR2  algorithm for second order optimality points}\label{appendixA}

Given $\epsilon, \epsilon_H>0$, Algorithm~\ref{fAR2_refresh_algo}  can be suitably modified in order to compute an $(\epsilon,\epsilon_H)$-approximately second-order minimizer (see \eqref{fomin2}). We will refer to this variant of \frozen\  as \frozenso\ (Frozen AR2-Second Order). The corresponding algorithms are detailed in Algorithm~\ref{fAR2_refresh_algo_so} and Algorithm~\ref{projected_secular_algo2} and the introduced
 modifications are listed below.
 \vskip 0.1in
\begin{itemize}
\item Once an   $\epsilon$-approximate first order optimality point has been reached, in order to obtain 
 further progress towards an approximate second-order optimality point, the stronger condition   $H_k+\widehat \lambda_k I\succ 0$ is needed to employ the  regularized Newton step, rather than \eqref{posedf_dir}.  
For the sake of simplicity, in Algorithm~\ref{fAR2_refresh_algo_so}
we did not distinguish  between iterations where 
$\|g_k\|>\epsilon$ or $\|g_k\|\le \epsilon$ and, 
at any iteration, in Line \ref{pd} we replace 
the condition $s_k^T(H_k+\widehat \lambda_k I) s_k \le 0 $   by   
$H_k+\widehat \lambda_k I\not \succ 0$.
\item Algorithm~\ref{fAR2_refresh_algo_so}  needs to be equipped with a second-order  stopping criterion;  in Step 1 the algorithm is stopped whenever 
the following two conditions hold:
\begin{equation}
\label{tc.2}
\| g_k\| \le \epsilon \quad \textrm{and}\quad \lambda_{\min} (H_k)\ge- \epsilon_H.
\end{equation}
\item {In Line \ref{acc2} of Algorithm~\ref{fAR2_refresh_algo_so}  the following condition has to be checked in addition to} \req{eq:dec2}:
\begin{equation} \label{stop_so}
\lambda_{\min} (\nabla^2 m_k(s_k))\ge- \theta_2 \|s_k \|, \quad \theta_2>0.
\end{equation}
\item In Line 11 of Algorithm~\ref{fAR2_refresh_algo_so}  $(s_k,\lambda_k)$ has to satisfy \eqref{stop_so} in addition to  \req{eq:dec1}-\req{eq:dec2}.
\item {In Line 8 of Algorithm \ref{projected_secular_algo2} the condition $\lambda_{\min} (\nabla^2 m_k(W^{(j)}\hat s))\ge- \theta_2 \|\hat s \|$ has to be checked, too.
}
\end{itemize}

\vskip 0.1in

\begin{algorithm}[tb]
\caption{The Frozen AR2  algorithm for second order optimality points (\frozenso)} \label{fAR2_refresh_algo_so}
\begin{algorithmic}[1]
\Require  An initial point $x_0$; accuracy thresholds $\epsilon, \epsilon_H \in (0,1)$;
an initial regularization parameter $\sigma_0 >0$ are given, and constants
$\eta_1, \eta_2, \gamma_1, \gamma_2,  \theta_1, \theta_2, \sigma_{\min} $ s.t.
$$\sigma_{\min} \in (0,\sigma_0],\ \theta_1 >0, \ \theta_2>0, \ 0<\eta_1\le\eta_2<1,$$
$$0<\gamma_1 < 1 < \gamma_2, \ 0<C_{{\rm low}}<C_{{\rm up}}
$$
An integer $j_{\max} \ll n$.

\State Compute $f(x_0), g_0 = \nabla f(x_0), H_0 = \nabla^2 f(x_0)$, and set $k=0$, $\mathtt{refresh}=1$ and $V_0=[]$.
 \State{\bf Step 1: Test for termination.} If $\|g_k\| \le \epsilon$ and 
$\lambda_{\min} (H_k)\ge- \epsilon_H$, terminate.
 \State{\bf Step 2: Step computation.}
\State \hskip 0.2in Invoke Algorithm~\ref{projected_secular_algo2} providing the scalar $\widehat \lambda_k$, and the reduced step $\widehat s_k$,
\Statex \hskip 0.2in  the basis $W_{k}$ and  $V_{k+1}$
\State \hskip 0.2in  {\bf if} $W_{k} \widehat s_k$ satisfies \req{eq:dec2} and \req{stop_so}
\label{acc2}
\State\hskip 0.4in     set $s_k = W_{k} \widehat s_k$ and {\bf goto} Step 3 \label{psub2}
\State \hskip 0.2in {\bf end if}
\State \hskip 0.2in Compute $s_k = - ( H_k+\widehat \lambda_k I)^{-1} g_k$. \label{pnew2} \hfill   {\tt \% regularized Newton step}
\State \hskip 0.2in    {\bf if} {$(H_k+\widehat \lambda_k I) \not \succ 0 $ or {$\|s_k\|/\|\widehat s_k\|<C_{{\rm low}}$ or $\|s_k\|/\|\widehat s_k\|>C_{{\rm up}}$}}  \label{pd2} 
\State \hskip 0.4in   {\bf if} {$\mathtt{refresh}$}\label{jmax2}
\State \hskip 0.6in   Find $ \lambda_k$ and
$s_k = -(H_k+ \lambda_k I)^{-1} g_k$ such that \eqref{eq:dec1}-\eqref{eq:dec2}, and \req{stop_so} 
\Statex \hskip 0.6in hold by applying the  secant method to {$\phi_R( \lambda_k; g_k, H_k, \sigma_k)=0$ }  
\State \hskip 0.6in  {\bf Goto} Step 3 \label{secant2} 
\State \hskip 0.4in   {\bf else}
\State \hskip 0.6in     Define $x_{k+1}= x_k$, $\sigma_{k+1}= \sigma_k$, set $\mathtt{refresh}=1$, $k=k+1$ \label{rejpd2}
\Statex{\hskip 0.6in {\bf Goto} Step 1}
\hfill  {\tt \% unsuccessful iteration, step rejection}
\State \hskip 0.4in {\bf end if}
\State \hskip 0.2in {\bf end if}
\State {\bf Step 3} Compute 
$$
\rho_k = \frac{f(x_k) - f(x_k+ s_k)}{T_k(0)-T_k(s_k)}.
$$
\State   {\bf Step 4} %:  Successful iteration}
\If{$\rho_k \ge \eta_1$} \hfill  {\tt \% Successful iteration}
\State
set $x_{k+1}= x_k +  s_k$, compute $g_{k+1} = \nabla f(x_{k+1})$, $H_{k+1} = \nabla^2 f(x_{k+1})$
\Else \hfill    {\tt \% unsuccessful iteration, step rejection}
\State set $x_{k+1}= x_k$
\EndIf

\State {\bf Step 5: Regularization parameter update.} Compute
   \begin{equation*}\label{sig2bis2}
   \sigma_{k+1} \in \left \{
   \begin{array}{lll}
    [\max\{\sigma_{\min}, \gamma_1 \sigma_k\}, \sigma_k  ]                   & \mbox{ if } \rho_k\geq \eta_2
     \\
   \sigma_k &\mbox{ if }\rho_k\in [\eta_1,\eta_2)
      \\%& \mbox{{\small[successful iteration]}}\\
     \gamma_2 \sigma_k & \mbox{ otherwise }      
   \end{array}
   \right .
    \end{equation*}
   \State Increment $k$ by one set $\mathtt{refresh}=0$ and {\bf goto} Step 1.

\end{algorithmic}
\end{algorithm}

%%%%%%%%%%
\begin{algorithm}[tb]
\caption{Minimization of $m_k$ over the low-dimensional subspace} (second order optimality case) \label{projected_secular_algo2}
\begin{algorithmic}[1]
\Require The matrix $H_k$; the vector $g_k$; the matrix $V_{k}\in\mathbb{R}^{n\times d_k}$; accuracy thresholds $\theta_1>0$, $\theta_2>0$; the parameter $\sigma_k$;
an integer $j_{\max} \ll n$; {\tt refresh}.

\State \hskip 0.2in {\bf if} $\mathtt{refresh}$ \hskip 1.7in {\tt\% Generate new proj space}

\State \hskip 0.4in Set $V_{k+1}=[\, ]$ 

\State \hskip 0.4in {\bf for} $j=1,\ldots,j_{\max}-1$

\State \hskip 0.6in Set $W^{(j)}={\rm orth}([V_{k+1},g_k])$
\State \hskip 0.6in Compute projections $g^{(j)} = (W^{(j)})^T g_k$, $\quad H^{(j)} = (W^{(j)})^T H_k W^{(j)}$
\State \hskip 0.6in Find $\widehat \lambda$ s.t.
{
 $\phi_R(\widehat \lambda; g^{(j)}, H^{(j)}, \sigma_k) =0,$ i.e.}
 \begin{equation*}
  \widehat \lambda=\sigma_k \|(H^{(j)}+\widehat \lambda I)^{-1} g^{(j)}\|
 \end{equation*}
 \State \hskip 0.6in Set
 $\widehat s = -(H^{(j)}+\widehat \lambda I)^{-1} g^{(j)}$
 \State \hskip 0.6in {\bf if} $\|\nabla m_k(W^{(j)}\widehat s) \| \le \frac{1}{2}\theta_1 \|\widehat s\|^2$ and 
 $$
 \lambda_{\min} (\nabla^2 m_k(W^{(j)}\widehat s))\ge- \theta_2 \|\widehat s \|
 $$
 \State \hskip 0.8in Set  
 $\widehat \lambda_k=\widehat \lambda$, $\widehat s_k=\widehat s$,  $W_k =W^{(j)}$  and {\bf return}
 
\State \hskip 0.6in {\bf end if}
\State \hskip 0.6in Expand $V_{k+1}$ with new Krylov direction~\label{alg:expand_krylov2}
\State \hskip 0.4in {\bf end for}
\State \hskip 0.4in Set 
 $\widehat \lambda_k=\widehat \lambda$, $\widehat s_k=\widehat s$,  {$W_k =W^{(j)}$}  and {\bf return} 
\State \hskip 0.2in {\bf else}
\hskip 2.2in {\tt\% Project onto the old space}

\State \hskip 0.4in Set $W^{(\widehat \jmath)}= {\rm orth}([V_{k}, g_k])$, \revdone{where $\widehat \jmath=d_k+1$}  %
\State \hskip 0.4in Compute projections $g^{(\widehat \jmath)} = (W^{(\widehat \jmath)})^T g_k$, $\quad H^{(\widehat \jmath)} = (W^{(\widehat \jmath)})^T H_k W^{(\REV{\widehat \jmath})}$ 
\State \hskip 0.4in Find $\widehat \lambda$ s.t.
{
 $\phi_R(\widehat \lambda; g^{(\widehat \jmath)}, H^{(\widehat \jmath)}, \sigma_k) =0,$ i.e.}
 \begin{equation*}
  \widehat \lambda=\sigma_k \|(H^{(\widehat \jmath)}+\widehat \lambda I)^{-1} g^{(\widehat \jmath)}\|
 \end{equation*}
 \State \hskip 0.4in Set
 $\widehat s_k = -(H^{(\widehat \jmath)}+\widehat \lambda I)^{-1} g^{(\widehat \jmath)}$, $\quad \widehat \lambda_k=\widehat \lambda$
 \State \hskip 0.4in Set $V_{k+1}=V_{k}$ and $W_k= W^{(\widehat \jmath)}$
\REV{and {\bf return} }
\State \hskip 0.2in {\bf end if}

\end{algorithmic}
\end{algorithm}

%%%%%%%%%%%%%%%%%%%%%%%%%%%%%%%%%%%%%%%%%%%%%%%%%%%
We stress that  $H_k$ is expected to be positive definite during  the final convergence phase of  the  \frozenso \ method and this ensures that condition \eqref{bound_p} holds eventually.

Next, we provide an upper bound on the number of iterations needed by \frozenso\ to compute an
 $(\epsilon,\epsilon_H)$-approximately second-order minimizer.

\begin{theorem}
Suppose that {\rm\textbf{AS.1}} and {\rm\textbf{AS.2}} hold.
Given $\epsilon, \epsilon_H \in (0,1)$, there exist  positive constants  $\kappa_p$ and $\kappa_{H}$ such that the \frozenso\ algorithm 
needs at most
$$
\left \lceil (f(x_0)- f_{\rm low})
 \max \left\{\frac{\epsilon^{-\frac{3}{2}}}{\kappa_p},\frac{\epsilon_H^{-2}}{\kappa_H} \right\}+1
\right \rceil,
$$
successful iterations and at most 
$$
 \left \lceil
 (f(x_0)- f_{\rm low})
 \max \left \{\frac{\epsilon^{-\frac{3}{2}}}{\kappa_p},\frac{\epsilon_H^{-2}}{\kappa_H}\right\}+1
\right \rceil
                \left(3/2+\frac{3/2|\log\gamma_1|}{\log\gamma_2}\right)+
\frac{3/2}{\log\gamma_2}\log\left(\frac{\sigma_{\max}}{\sigma_0}\right), 
$$
iterations to compute an iterate $x_k$ such that $\|g_k\|\le \epsilon$ and $\lambda_{{\min}}
(H_{k})\ge-\epsilon_H$.
\end{theorem}

\begin{proof}
Let us keep the same notation of  Section 4 and denote with 
${\cal I}_{RN}$ the set of iteration indexes such that $s_k$ has been computed  in Line \ref{pnew2} of Algorithm \ref{fAR2_refresh_algo_so}.
Note that at  any successful iteration $k$ in such a set, it must hold that $H_k +\hat \lambda I$ is positive definite.  Then, 
$$
\lambda_{\min} (H_k)>-\widehat \lambda.
$$
Therefore, by using \textbf{AS.1} and Lemma 3.3,  for any successful iteration $k$ with $k\in {\cal I}_{RN}$  and for any  $d\in \IR^n$, $\|d\|=1$,  we obtain
\begin{eqnarray*}
d^T H_{k+1} d&=&d^T (H_{k+1}-H_k) d+ d^T H_k d\ge -L_2  \|s_k\|- \widehat \lambda \\
&=&  -L_2   C_{{\rm up}} \|\widehat s_k\|- \sigma_k \|\widehat s_k\|\ge -(L_2 C_{{\rm up}}+ \sigma_{\max})  \|\widehat s_k\|.
\end{eqnarray*}
This latter inequality, along with 
$\min_{\|d\|=1} d^T H_{k+1} =\lambda_{{\min}} (H_{k+1})$ and 
Lemma  \ref{lemma_decrease}, yields 
\begin{equation} \label{bound_p2}
\|\widehat s_k\|\ge \max \{ \kappa_0^{1/3} \|g_{k+1}\|^{1/2}, -\kappa_{H,0}^{1/3} \lambda_{{\min}} (H_{k+1}) \},
\end{equation} 
where
$\kappa_{H,0}=\left(\frac{1}{L_H C_{{\rm up}}+ \sigma_{\max}}\right)^3$ and $\kappa_0$ has been defined as in Lemma \ref{lemma_decrease}.
Moreover, by \eqref{eq:diffT0}
$$
f(x_k)-f(x_{k+1})\ge \eta_1 (T_k(0)-T_k(s_k)) 
\ge \frac{\eta_1}{2} \sigma_k \|\widehat s_k\|\|s_k\|^2 
\ge \frac{\eta_1}{2} \sigma_{\min} C_{\rm low}^2  \|\widehat s_k\|^3,
$$ and we can conclude 
 that for any successful iteration $k$ with $k\in {\cal I}_{RN}$ it holds
$$
f(x_k)-f(x_{k+1})\ge  \frac{\eta_1}{2} \sigma_{\min} C_{\rm low}^2   \max \{ \kappa_0  \|g_{k+1}\|^{3/2}, -\kappa_{H,0} \lambda_{\min} (H_{k+1})^{3} \}.
$$
In case $ k \not \in {\cal I}_{RN}$,  Lemma 3.3.3 in  \cite{book_compl} ensures that the step $s_k$ computed in Line \ref{psub} of Algorithm \ref{fAR2_refresh_algo_so} satisfies:
$$
\|s_k\|\ge -\frac {1} {L_H+\theta_2+\sigma_{\max}} \lambda_{\min} (H_{k+1}) .
$$
Then, from  \eqref{eq:diffT1} and Lemma \ref{lemma_decrease} it follows
that
\begin{eqnarray*}
f(x_k)-f(x_{k+1})&\ge& \eta_1 (T_k(0)-T_k(s_k))\ge \frac{\eta_1}{3} \sigma_k\|s_k\|^3\\
    &\ge& \frac{\eta_1}{3}  \sigma_{\min} \max\left\{\kappa_1 \|g_{k+1}\|^{3/2},  -\kappa_{H,1} \lambda_{\min} (H_{k+1})^{3}\right\} ,
\end{eqnarray*}
where $\kappa_{H,1}= \left (\frac {1} {L_H+\theta_2+\sigma_{\max}} \right)^{1/3}$.

Finally, 
 for any iteration $ k$  such that the \frozenso\ algorithm has not terminated before or at
iteration $k + 1$, we must have that either  $\|g_{k+1}\|>\epsilon$ or $ \lambda_{\min} (H_{k+1})< -\epsilon_H$. 
 Therefore, at any successful iteration $k$ we have 
  $$f(x_k)-f(x_{k+1})\ge  \eta_1  \sigma_{\min} \min \left\{\kappa_p \epsilon^{3/2},  \kappa_{H} \epsilon_{H} ^{3}\right\}, 
$$
with $\kappa_p=\eta_1\sigma_{\min} \min\left\{C_{\rm low}^2\frac{\kappa_0}{2},\frac{\kappa_1}{3}\right\} $ and $\kappa_H=\eta_1\sigma_{\min}\min\left\{C_{\rm low}^2\frac{\kappa_{H,0}}{2},\frac{\kappa_{H,1}}{3}\right\}$.
The thesis then follows proceeding as in the proof of Theorem \ref{theorem_complexity}.
\end{proof}

\section{The {\sc OPM} test set and complete results} \label{sec:opm}
We report in Table \ref{tab:OPM} the {\sc OPM} test problems used in the experiments described in section \ref{expe:OPM} and their dimension.
Tables \ref{tab:risuopm} and \ref{tab:risuopm2}  show the results obtained by using \galr, \frozenr\ and \frozenp\ in the solution of 
the OPM problems listed in Table \ref{tab:OPM}.
The tables display the number of nonlinear iterations (\#NLI) along with the number of  factorizations (\#fact) for the three solvers. Moreover, in case of \frozenr\ and \frozenp, they report also the number of iterations where the basis is refreshed (\#ref), the average dimension of the projected problems ($ave_K$), the number of times
the step $s_k$ is computed in Line \ref{psub} of Algorithm \ref{fAR2_refresh_algo} (\#sub) using a frozen subspace, and the number of times the solvers employ the secant method in Line~\ref{secant} of 
Algorithm \ref{fAR2_refresh_algo} (\#sec). The symbol '*' means that the nonlinear solver did not converge in 5000 iterations.

\begin{table}[tb]
\scriptsize
  \centering
    \begin{tabular}{l|r|l|r|l|r|l| r}
    \toprule
    Problem   & \multicolumn{1}{c|}{$n$ } & Problem   & \multicolumn{1}{c|}{$n$ } & Problem   & \multicolumn{1}{l|}{$n$ } & \multicolumn{1}{l|}{Problem  } & \multicolumn{1}{c}{$n$ } \\
    \midrule
         ARGLINA    	 & 1000  &       DIXMAANC   	 & 3000  &       ENGVAL1    	 & 1000  &      PENALTY3   	 & 1000\\
    
          ARGLINB    	 & 1000  &       DIXMAAND   	 & 3000  &       ENGVAL2    	 & 1000  &      POWELLSG   	 & 1000\\
    
          ARGLINC    	 & 1000  &       DIXMAANE   	 & 3000  &       EXTROSNB   	 & 1000  &      POWR       	 & 1000\\
    
          ARGTRIG    	 & 1000  &       DIXMAANF   	 & 3000  &       FMINSURF   	 & 900   &      ROSENBR    	 & 1000\\
    
          ARWHEAD    	 & 1000  &       DIXMAANG   	 & 3000  &       FREUROTH   	 & 1000  &      SCOSINE    	 & 1000\\
    
          BDARWHD    	 & 1000  &       DIXMAANH   	 & 3000  &       GENHUMPS   	 & 1000  &      SCURLY10   	 & 1000\\
    
          BROWNAL    	 & 1000  &       DIXMAANI   	 & 3000  &       HELIX      	 & 1000  &      SCURLY20   	 & 1000\\
    
          BROYDENBD  	 & 1000  &       DIXMAANJ   	 & 3000  &       HILBERT    	 & 1000  &      SCURLY30   	 & 1000\\
    
          CHANDHEU   	 & 1000  &       DIXMAANK   	 & 3000  &       INDEF      	 & 1000  &      SENSORS    	 & 1000\\
    
          CHEBYQAD   	 & 1000  &       DIXMAANL   	 & 3000  &       INTEGREQ   	 & 1000  &      SPMSQRT    	 & 1000\\
    
          COSINE     	 & 1000  &       DIXON      	 & 1000  &       MANCINO    	 & 1000  &      TQUARTIC   	 & 1000\\
    
          CRGLVY     	 & 1000  &       DQRTIC     	 & 1000  &       MSQRTALS   	 & 900   &      TRIDIA     	 & 1000\\
    
          CUBE       	 & 1000  &       EDENSCH    	 & 1000  &       MSQRTBLS   	 & 900   &      VARDIM     	 & 1000\\
    
          CURLY10    	 & 1000  &       EG2        	 & 1000  &       NONDIA     	 & 1000  &      WMSQRTALS  	 & 900\\
    
          CURLY20    	 & 1000  &       EG2S       	 & 1000  &       NONDQUAR   	 & 1000  &      WMSQRTBLS  	 & 900\\
    
          CURLY30    	 & 1000  &       EIGENALS   	 & 1056  &       NZF1       	 & 1300  &      WOODS      	 & 1000\\
    
          DIXMAANA   	 & 3000  &       EIGENBLS   	 & 1056  &       PENALTY1   	 & 1000  &        	       &     \\
         DIXMAANB   	 & 3000  &       EIGENCLS   	 & 1056  &       PENALTY2   	 & 1000  &         	      &     \\
   
\bottomrule\end{tabular}%
      \caption{The {\sc OPM} test problems and their dimension.\label{tab:OPM}}
 \end{table}%

\begin{sidewaystable} 
\scriptsize
  \centering							
									
   \begin{tabular}{l|rr|rrrrrr|rrrrrr}
   \toprule
        \multicolumn{1}{r}{} & \multicolumn{2}{c}{\galr} & \multicolumn{6}{c}{\frozenr}                           & \multicolumn{6}{c}{\frozenp} \\
        \midrule
         & \#NLI & \#fact &  \#NLI & \#fact & \# ref   & $ave_K$ & \#sub & \#sec & \#NLI & \#fact & \#ref   & $ave_K$ & \#sub & \#sec \\ \midrule
     ARGLINA & 5     & 5     & 5     & 1     & 1     & 1     & 4     & 0     & 5     & 0     & 1     & 1     & 4     & 0 \\
        ARGTRIG & 12    & 24    & 12    & 14    & 1     & 3     & 0     & 0     & 12    & 11    & 1     & 17    & 0     & 0 \\
        ARWHEAD & 5     & 8     & 5     & 2     & 1     & 2     & 4     & 0     & 5     & 0     & 1     & 2     & 4     & 0 \\
        BDARWHD & 12    & 24    & 12    & 2     & 1     & 2     & 11    & 0     & 12    & 0     & 1     & 2     & 11    & 0 \\
        BROWNAL & 2     & 4     & 2     & 2     & 1     & 1     & 0     & 0     & 2     & 0     & 1     & 1     & 1     & 0 \\
        BROYDENBD & 7     & 11    & 7     & 9     & 1     & 3     & 0     & 0     & 7     & 6     & 1     & 3     & 0     & 0 \\
        CHANDHEU & 2     & 2     & 2     & 2     & 1     & 2     & 1     & 0     & 2     & 0     & 1     & 2     & 1     & 0 \\
       % CHEBYQAD &* &* &* &* &* &* &* &* &* &* &* &* &* &* \\
        CRGLVY & 11    & 22    & 12    & 12    & 1     & 2     & 1     & 0     & 13    & 10    & 1     & 2     & 2     & 0 \\
        CUBE & 47    & 161   & 47    & 25    & 1     & 3     & 24    & 0     & 47    & 8     & 1     & 10    & 38    & 0 \\
        CURLY10 & 23    & 109   & 21    & 77    & 2     & 40.7  & 2     & 1     & 20    & 16    & 3     & 20.5  & 1     & 0 \\
        CURLY20 & 24    & 118   & 25    & 92    & 3     & 43.2  & 3     & 1     & 24    & 20    & 3     & 24    & 1     & 0 \\
        CURLY30 & 26    & 133   & 31    & 166   & 4     & 43.9  & 1     & 2     & 29    & 25    & 3     & 21.7  & 1     & 0 \\
        DIXMAANA & 6     & 10    & 8     & 2     & 1     & 1     & 6     & 0     & 8     & 1     & 1     & 1     & 6     & 0 \\
        DIXMAANB & 38    & 161   & 57    & 58    & 2     & 2.1   & 2     & 0     & 11    & 2     & 1     & 2     & 8     & 0 \\
        DIXMAANC & 21    & 66    & 8     & 3     & 1     & 2     & 6     & 0     & 8     & 1     & 1     & 2     & 6     & 0 \\
        DIXMAAND & 28    & 129   & 42    & 41    & 1     & 2     & 2     & 0     & 9     & 1     & 1     & 2     & 7     & 0 \\
        DIXMAANE & 10    & 31    & 10    & 3     & 1     & 2     & 8     & 0     & 10    & 1     & 1     & 2     & 8     & 0 \\
        DIXMAANF & 75    & 320   & 64    & 76    & 2     & 11.2  & 2     & 0     & 38    & 30    & 1     & 2     & 7     & 0 \\
        DIXMAANG & 67    & 395   & 69    & 54    & 1     & 2     & 16    & 0     & 35    & 23    & 2     & 2.2   & 10    & 0 \\
        DIXMAANH & 73    & 334   & 54    & 53    & 2     & 2.6   & 5     & 0     & 46    & 32    & 1     & 2     & 13    & 0 \\
        DIXMAANI & 12    & 42    & 21    & 12    & 1     & 2     & 10    & 0     & 14    & 2     & 1     & 2     & 11    & 0 \\
        DIXMAANJ & 81    & 430   & 69    & 57    & 2     & 2.2   & 15    & 0     & 36    & 26    & 1     & 2     & 9     & 0 \\
        DIXMAANK & 59    & 369   & 77    & 66    & 3     & 2.2   & 15    & 0     & 40    & 30    & 1     & 2     & 9     & 0 \\
        DIXMAANL & 93    & 445   & 53    & 52    & 1     & 2     & 2     & 0     & 54    & 43    & 2     & 2.4   & 9     & 0 \\
        DIXON & 9     & 32    & 6     & 7     & 1     & 3     & 1     & 0     & 6     & 4     & 1     & 4     & 1     & 0 \\
        DQRTIC & 8     & 8     & 8     & 1     & 1     & 1     & 7     & 0     & 8     & 0     & 1     & 0     & 7     & 0 \\
        EDENSCH & 24    & 89    & 13    & 8     & 1     & 2     & 6     & 0     & 13    & 6     & 1     & 2     & 6     & 0 \\
        EG2 & 16    & 32    & 9     & 1     & 1     & 1     & 8     & 0     & 9     & 0     & 1     & 1     & 8     & 0 \\
        EG2S & 68    & 207   & 72    & 61    & 2     & 3.7   & 14    & 0     & 99    & 93    & 2     & 3.8   & 4     & 0 \\
        EIGENALS & 523   & 5309  & 89    & 97    & 1     & 10    & 1     & 0     & 65    & 61    & 1     & 8     & 3     & 0 \\
  EIGENBLS	& 670	& 2289 &	847 & 	894	&	8	&	6.9&	0	&0 &	739 &	734	&	9	&	45.0 &	0	& 4 \\

        \bottomrule
    \end{tabular}%
    \caption{Complete results obtained by \galr\, \frozenr\ and \frozenp on OPM problems in Table \ref{tab:OPM}.}
      \label{tab:risuopm}%
%\end{table}%
\end{sidewaystable}

\begin{sidewaystable} 
\scriptsize
  \centering							
									
   \begin{tabular}{l|rr|rrrrrr|rrrrrr}
   \toprule
        \multicolumn{1}{r}{} & \multicolumn{2}{c}{\galr} & \multicolumn{6}{c}{\frozenr}                           & \multicolumn{6}{c}{\frozenp} \\
        \midrule
         & \#NLI & \#fact &  \#NLI & \#fact & \# ref   & $ave_K$ & \#sub & \#sec & \#NLI & \#fact & \#ref   & $ave_K$ & \#sub & \#sec \\ \midrule
        EIGENCLS & 768   & 6398  & 1016  & 1048  & 3     & 15.6  & 4     & 0     & 1273  & 1260  & 3     & 49.4  & 11    & 1 \\
        ENGVAL1 & 8     & 13    & 8     & 6     & 1     & 1     & 2     & 0     & 8     & 5     & 1     & 1     & 2     & 0 \\
        EXTROSNB & 10    & 18    & 10    & 12    & 1     & 3     & 0     & 0     & 10    & 9     & 1     & 6     & 0     & 0 \\
        FMINSURF & 94    & 211   & 92    & 88    & 1     & 3     & 6     & 0     & 112   & 106   & 1     & 4     & 5     & 0 \\
        FREUROTH & 6     & 13    & 6     & 7     & 1     & 2     & 0     & 0     & 6     & 5     & 1     & 2     & 0     & 0 \\
        HELIX & 17    & 60    & 17    & 17    & 1     & 2     & 1     & 0     & 17    & 10    & 1     & 2     & 6     & 0 \\
        HILBERT & 6     & 12    & 7     & 2     & 1     & 1     & 5     & 0     & 7     & 1     & 1     & 1     & 5     & 0 \\
        %INDEF & 19    & 93    & 51    & 35    & 3     & 3.1   & 25    & 0     &* &* &* &* &* &* \\
        INDEF & 19    & 93    & 51    & 35    & 3     & 3.1   & 25    & 0     &196 & 55 & 13 & 2.3 &128 &0 \\ %try
        INTEGREQ & 15    & 30    & 13    & 16    & 1     & 4     & 0     & 0     & 13    & 13    & 1     & 50    & 0     & 1 \\
        MANCINO & 121   & 884   & 30    & 555   & 11    & 46    & 0     & 8     & 30    & 50    & 11    & 46.5  & 0     & 9 \\
        MSQRTALS & 241   & 2526  & 412   & 625   & 23    & 8.6   & 6     & 0     & 256   & 250   & 16    & 45.7  & 2     & 12 \\
        MSQRTBLS & 318   & 3945  & 421   & 774   & 41    & 11    & 14    & 0     & 468   & 485   & 31    & 47    & 2     & 29 \\
        NONDIA & 38    & 83    & 38    & 2     & 1     & 2     & 37    & 0     & 38    & 0     & 1     & 2     & 37    & 0 \\
        NONDQUAR & 12    & 58    & 12    & 18    & 1     & 7     & 0     & 0     & 12    & 11    & 1     & 11    & 0     & 0 \\
        NZF1 & 10    & 20    & 10    & 8     & 1     & 3     & 4     & 0     & 10    & 5     & 1     & 9     & 4     & 0 \\
        PENALTY1 & 14    & 28    & 12    & 1     & 1     & 1     & 11    & 0     & 12    & 0     & 1     & 1     & 11    & 0 \\
      %  PENALTY2 &* &* &* &* &* &* &* &* &* &* &* &* &* &* \\
        PENALTY3 & 838   & 3296  & 773   & 835   & 6     & 8.4   & 3     & 0     & 75    & 69    & 6     & 35.5  & 3     & 3 \\
        POWELLSG & 12    & 31    & 12    & 4     & 1     & 3     & 10    & 0     & 12    & 0     & 1     & 3     & 11    & 0 \\
        POWR & 13    & 26    & 12    & 11    & 1     & 11    & 11    & 0     & 12    & 12    & 1     & 50    & 0     & 1 \\
        ROSENBR & 3851  & 7269  & 4617  & 4619  & 1     & 3     & 0     & 0     & 4596  & 4595  & 1     & 6     & 0     & 0 \\
        SENSORS & 54    & 318   & 61    & 232   & 5     & 29    & 0     & 2     & 79    & 68    & 4     & 22.6  & 7     & 0 \\
        SPMSQRT & 18    & 59    & 18    & 24    & 2     & 4.4   & 1     & 0     & 25    & 24    & 1     & 4     & 0     & 0 \\
        TQUARTIC & 12    & 24    & 12    & 9     & 1     & 4     & 6     & 0     & 12    & 5     & 1     & 32    & 6     & 0 \\
        TRIDIA & 5     & 10    & 5     & 4     & 1     & 2     & 2     & 0     & 5     & 2     & 1     & 2     & 2     & 0 \\
        VARDIM & 24    & 252   &* &* &* &* &* &* &* &* &* &* &* &* \\
        WMSQRTALS & 89    & 761   & 293   & 393   & 12    & 7.9   & 2     & 0     & 292   & 280   & 18    & 47.3  & 3     & 9 \\
        WMSQRTBLS & 183   & 1387  & 240   & 345   & 13    & 10.5  & 3     & 0     & 352   & 348   & 7     & 46.1  & 2     & 5 \\
        WOODS & 15    & 91    & 15    & 4     & 1     & 4     & 14    & 0     & 15    & 0     & 1     & 4     & 14    & 0 \\
        \bottomrule
    \end{tabular}%
    \caption{Complete results obtained by \galr\, \frozenr\ and \frozenp on OPM problems in Table \ref{tab:OPM}. (continued)}
      \label{tab:risuopm2}%
%\end{table}%
\end{sidewaystable}

\bibliographystyle{siam}
\bibliography{secularbib}

\end{document}